\documentclass[10pt, 5p]{article} % use larger type; default would be 10pt

\usepackage[utf8]{inputenc} % set input eNoding (not needed with XeLaTeX)
\usepackage{amsmath}
\usepackage{amsfonts}
\usepackage{amssymb}
\usepackage{amsthm}
\usepackage{a4wide}
\usepackage{booktabs}
\usepackage{url}
\usepackage[square,numbers]{natbib}
\usepackage{graphicx} % support the \iNludegraphics command and options
\usepackage[parfill]{parskip} % Activate to begin paragraphs with an empty line rather than an indent
\newtheorem{lemma}{Lemma}

\newtheorem{corollary}{Corollary}
\newtheorem{theorem}{Theorem}
\everymath{\displaystyle}

\usepackage{amssymb} %pour 'align' etc
\everymath{\displaystyle}
\usepackage{xcolor}

\newtheorem{observation}{Observation}
\usepackage{amsmath}
\usepackage{amsfonts}
\usepackage{amssymb}

\date{December 2016}
\title{Revisiting minimum profit conditions in uniform price day-ahead electricity auctions}
\author{Mehdi Madani\footnote{Louvain School of Management - Place des Doyens 1 bte L2.01.01, 1348 Louvain-la-Neuve, Belgium. Email: mehdi.madani@uclouvain.be}         \and
        Mathieu Van Vyve\footnote{CORE \& Louvain School of Managment, 1348 Louvain-la-Neuve, Belgium. This author is also member of ECORE, the association between CORE and ECARES.}
}

\begin{document}

% \begin{center}
% CORE DISCUSSION PAPER
% 
% 2015/??
% \end{center}

\vspace{-0.2cm}
{\let\newpage\relax\maketitle}

\begin{abstract}

We examine the problem of clearing day-ahead electricity market auctions where each bidder, whether a producer or consumer, can specify a minimum profit or maximum payment condition constraining the acceptance of a set of bid curves spanning multiple time periods in locations connected through a  transmission network with linear constraints. Such types of conditions are for example considered in the Spanish and Portuguese day-ahead markets. This helps describing the recovery of start-up costs of a power plant, or analogously for a large consumer, utility reduced by a constant term. A new market model is proposed with a corresponding MILP formulation for uniform locational price day-ahead auctions, handling bids with a minimum profit or maximum payment condition in a uniform and computationally-efficient way. An exact decomposition procedure with sparse strengthened Benders cuts derived from the  MILP formulation is also proposed. The MILP formulation and the decomposition procedure are similar to computationally-efficient approaches previously proposed to handle so-called block bids according to European market rules, though the clearing conditions could appear different at first sight. Both solving approaches are also valid to deal with both kinds of bids simultaneously, as block bids with a minimum acceptance ratio, generalizing fully indivisible block bids,  are but a special case of the MP bids introduced here. We argue in favour of the MP bids by comparing them to previous models for minimum profit conditions proposed in the academic literature, and to the model for minimum income conditions used by the Spanish power exchange OMIE.

%, now part of the pan-European PCR market

%\textbf{Keywords: }{Day-ahead electricity market auctions, \and Non-convexities, \and Mixed Integer Programming, \and Market Coupling, \and Equilibrium Prices, \and minimum profit conditions, \and maximum payments conditions, \and MP conditions }
%\PACS{PACS code1 \and PACS code2 \and more}

%\textbf{Mathematics Subject Classification (2000): }{90C11 \and 90-08 \and 90C06 }

\end{abstract}

%\textbf{Keywords:} Integer programming, OR in energy, Auctions/bidding, Large scale optimization

%{\let\newpage\relax\maketitle}
\maketitle

\newpage

\section{Introduction}
\label{intro}

\subsection{Minimum profit conditions and Near-Equilibrium in non-convex day-ahead electricity auctions}\label{subsec:non-convex-markets}

Day-ahead electricity markets are organized markets where electricity is traded for the  24 hours of the next day. They can take the form of single or two sided auctions (pool with mandatory participation to match forecast demand or auctions confronting elastic offer and demand). The prices set in day-ahead markets are used as reference prices for many electricity derivatives, and such markets are taking more  importance with the ongoing liberalization and coupling of  electricity markets around the world in general, and in Europe in particular.

Clearing these auctions amounts to finding - ideally- a partial equilibrium using submitted bids describing demand and offer profiles, depending on the utility, production costs and operational constraints of market participants. A market operator, typically power exchanges in Europe, is in charge of computing a market clearing solution. 

It is well-known that for a well-behaved convex welfare optimization problem where strong duality holds, duality theory provides with equilibrium prices. However, to describe their operational constraints or cost structure, participants can specify for example a minimum output level of production (indivisibilities), or that the revenue generated by the traded power at the market clearing prices should cover some start-up costs if the plant is started. Similar bids could be specified for the demand side. This leads to the study of partial market equilibrium with uniform prices where indivisibilities and fixed costs must be taken into account, deviating from a well-behaved convex configuration studied in classical microeconomic textbooks, e.g. in \cite{mascolell}. The need for bidding products introducing non-convexities is due in particular to the peculiar nature of electricity and the non-convexities of production sets of the power plants. 

When considering a market clearing problem with non-convexities such as indivisibilities (so-called block bids in the Pan-European PCR market \cite{euphemia}), or start-up cost recovery conditions (so-called complex bids with a minimum income condition also called MIC bids in PCR), most of the time no market equilibrium exists, see e.g. the toy example in Section \ref{sec:toyexample} for an instance involving MIC bids, and in \cite{madani2015mip} for an instance involving block bids.

Let us also mention that in coupled day-ahead electricity markets, representation of the network is a particularly important matter. Aside the potential issues due to the simplifications or approximations made to represent a whole network, it is of main importance for participants to understand clearly the reason for price differences occurring between different locations. Economically speaking, locational prices should ideally form a spatial equilibrium, as historically studied  in \cite{enke, samuelson1952}, which could equivalently be interpreted as requiring optimality conditions for TSOs, relating locational price differences to the scarcity and marginal prices of transmission resources.

Near-equilibrium under minimum profit conditions in uniform price day-ahead electricity auctions is the main topic of the present contribution, and is also considered in references \cite{euphemia, GarciaBertrand, GarciaBertrand2, gabriel2013, ruiz2012}, which are discussed in Section \ref{sec:literature} below.

\subsection{Contribution and structure of this article} \label{subsec:contrib}

The main contribution of the present paper is to show how to handle minimum profit (or maximum payment) conditions in a new way which turns out to generalize both block orders with a minimum acceptance ratio used in France, Germany or Belgium,  and, mutatis mutandis, complex orders with a minimum income condition used in Spain and Portugal. The new approach consists in new bids, which we call MP bids (for minimum profit or maximum payment), and the corresponding mathematical programming formulation is a MILP modelling all the corresponding market clearing conditions without \emph{any} auxiliary variables, similar to an efficient MIP formulation previously proposed for block orders \cite{Madani2014}. An efficient Benders decomposition with sparse strenghtened cuts similar to the one proposed in  \cite{Madani2014} is  also derived. These MP bids hence seem an appropriate tool to foster market design and bidding products convergence among the different regions which form the coupled European day-ahead electricity markets of the Pan-European PCR project.  

We start by providing in Section \ref{sec:toyexample}
  a toy example illustrating the key points dealt with in the reminder of the article. It illustrates the issues arising when considering minimum profit conditions, and  alternatives to take them into account in the computation of market clearing solutions. We describe in Section \ref{sec:unrestricted-welfare} the notation used and a basic 'unrestricted' welfare maximization problem where such minimum profit conditions are first not enforced, also recalling the nice equilibrium properties which would hold in a convex market clearing setting. 
  
Section \ref{sec-modelling-MP} is devoted to modelling minimum profit conditions or more generally MP conditions, as with the approach proposed, the statement of a maximum payment condition for demand-side orders is formally identical. After reviewing previous contributions considering minimum profit conditions,  we derive economic interpretations for optimal dual variables of a welfare maximization program where an \emph{arbitrary} MP bids combination has been specified. We then develop the core result, showing how to consider MP bids in a computationally-efficient way, relying on previous results to provide a MILP formulation without complementarity constraints nor any auxiliary variable to model these MP conditions. Section \ref{sec:rampingconstr} shows how to adapt all results when ramping constraints of power plants are considered.

Section \ref{sec:Benders} derives from the MILP formulation provided in Section \ref{sec-modelling-MP} a Benders decomposition procedure with locally strengthened Benders cuts. These cuts are valid in subtrees of a branch-and-bound solving a primal welfare maximization program, rooted at nodes where an incumbent should be rejected because no uniform prices exist such that MP conditions are all satisfied. They complement the classical Benders cuts which we show to correspond indeed to 'no-good cuts' basically rejecting the current MP bids combination, and which are globally valid.

Numerical experiments are presented in Section \ref{sec-numerical}. Implementations have been made in Julia/JuMP \cite{LubinDunningIJOC} and are provided together with sample datasets in an online Git repository \cite{revisiting_mp_conditions}. They show the efficiency and merit of the new approach, in particular compared to the current practice in OMIE-PCR.

\section{Near-equilibrium and minimum profit conditions}\label{subsec:unrestricted-welfare}
\label{sec:notation}

\subsection{Position of the problem: a toy example}\label{sec:toyexample}

In the following toy example, a bid curve (in blue) represents some elastic demand. To satisfy this demand, there are two offer bids from two plants, each having different start-up costs (100 EUR and 200 EUR respectively), but the same marginal cost of 10 EUR / MW. Both plants bid their marginal cost curve and their start-up cost to the auctioneer.

\begin{figure}
\begin{center}
\includegraphics[width=\textwidth]{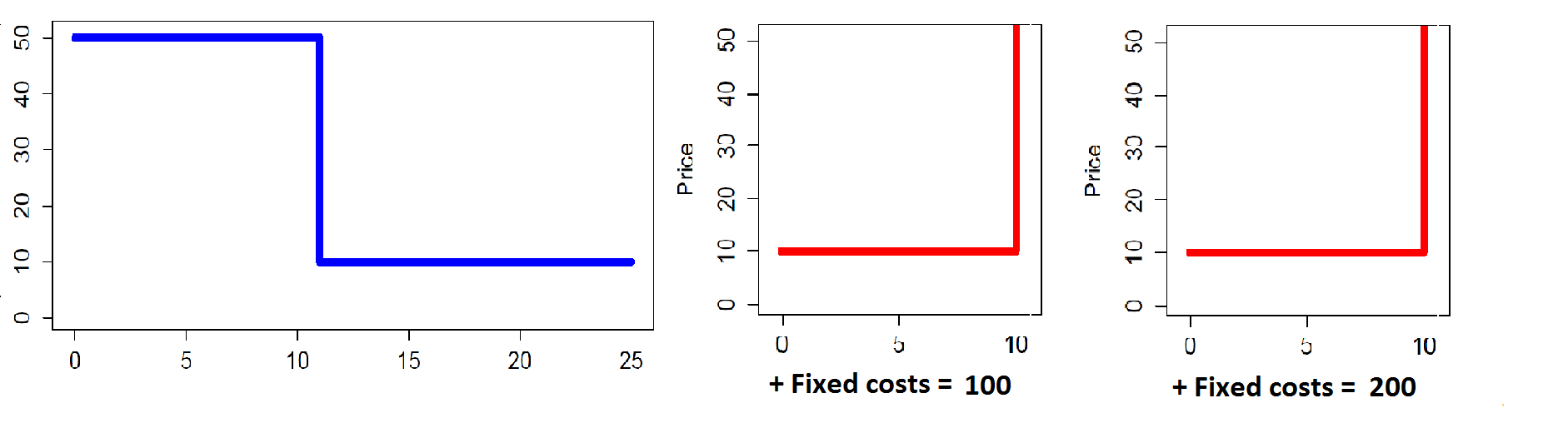}
\end{center}
\caption{Marginal cost/utility curves (see Table \ref{table1} for related start-up costs)}
\label{figure-1}
\end{figure}

Neglecting first the minimum income conditions stating that all costs should be recovered for online plants (i.e. both start-up and marginal costs), we can clear the market auction by matching the aggregated marginal costs (resp. utility) bid curves, as done in the left part of Figure \ref{figure-2}. In that case, the determined market clearing price would be 10 EUR /MW, and obviously, both power plants won't recover their costs for that market clearing price. 

However, if we allow the potentially paradoxical rejection of bids involving startup-costs, which is also tolerated in all previous propositions considering minimum profit conditions exposed in \cite{euphemia, gabriel2013, GarciaBertrand,GarciaBertrand2, ruiz2012}, then a 'satisfactory solution' could be obtained by either rejecting bid B or bid C. In that case, matching marginal cost/utility curves as in the right Figure \ref{figure-2}, we see that the market clearing price will rise to 50 EUR / MW and that, whatever the chosen offer B or C, the corresponding plant will recover all its  costs. Similar examples could be given for demand bids with a maximum payment condition.

These observations help understanding why it is not possible to get a market equilibrium such that all MP conditions are satisfied. It may be required to expel some bids from the market clearing solution that would be profitable for the market clearing prices obtained in that situation. On the other hand, including such 'paradoxically rejected bids' would modify prices such that the MP condition of some bid would not be satisfied any more.

\begin{figure}
\begin{center}
\includegraphics[scale=0.3]{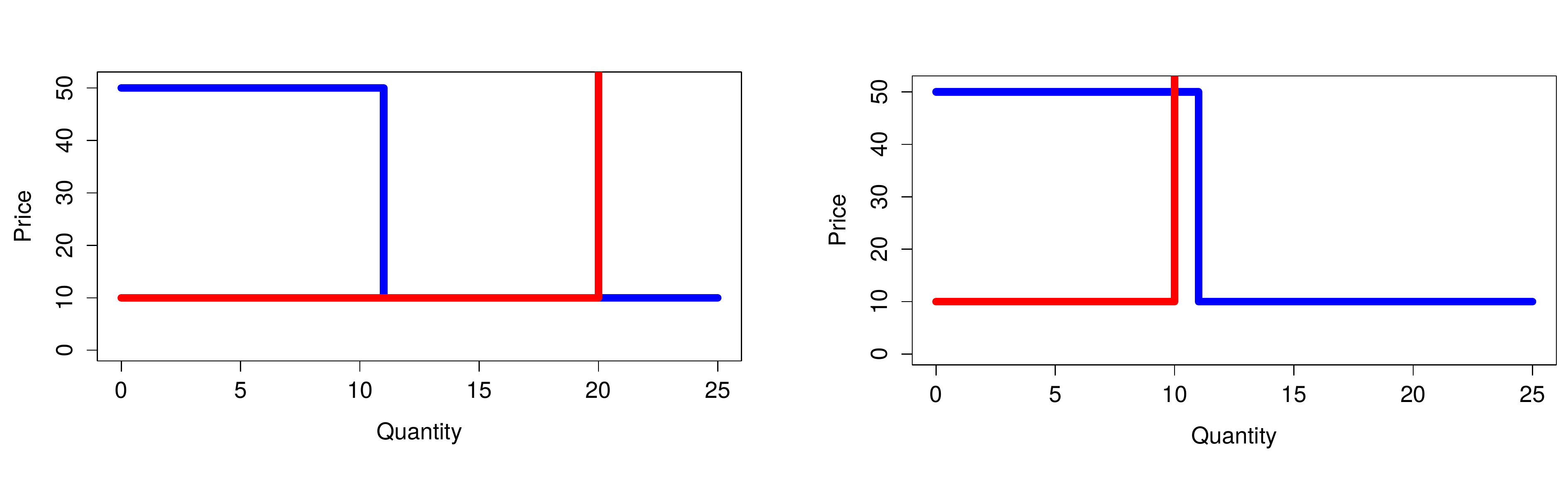}
\end{center}
\caption{Matching MP bids}
\label{figure-2}
\end{figure}

\begin{table}[h]
\begin{center}
\begin{tabular}{c|c|c|c}

Bids & Power (MW) & Limit price  (EUR/MW) & Start-up costs \\ 
\hline 
D1:  Demand bid 1 & 11 & 50 & - \\ 

D2:  Demand bid 2 & 14 & 10 & - \\ 

MP1:  Offer MP bid 1 & 10 & 10 & 100 \\ 

MP2:  Offer MP bid 2 & 10 & 10 & 200 \\ 

\end{tabular} 
\end{center}
\caption{Toy market clearing instance}
\label{table1}
\end{table}

\begin{table}[h]
\begin{center}
\begin{tabular}{c|c|c|c|c|c|c|c|}
   & Market Price & \multicolumn{2}{|c|}{Revenue} & \multicolumn{2}{|c|}{Costs}  & \multicolumn{2}{|c|}{Profits}\\ 
\hline 
   &              & MP1   & MP2       & MP1   & MP2    & MP1     & MP2  \\
\hline 
Matching MP1 \& MP2 & 10 & 100 & 100 & 200 & 300 & -100 & -200 \\ 
Matching MP1 & 50 & 500 & 0 & 200 & 0 & 300 & 0 \\ 
Matching MP2 & 50 & 0 & 500 & 0& 300 & 0 & 200 \\  
\end{tabular} 
\end{center}
\caption{Market outcomes}
\label{table2}
\end{table}

The second point is that, even if in both matchings the costs are recovered for the chosen plant, both matchings are not equivalent from a welfare point of view if we include fixed costs in the computation of the welfare.

Under current OMIE-PCR market rules, both matching possibilities are not distinguished because fixed costs are not included in the welfare maximizing objective function which only considers marginal costs (resp. utility) of selected plants (resp. consumers). In such a case, welfare is considered to be 400 whatever the chosen matching. Let us note that in the same way, in \cite{GarciaBertrand}, the fixed costs that should be recovered are not included in the welfare objective.

If we pay attention to fixed costs when computing welfare, matching MP1 yields a welfare of 300 while matching MP2 yields a welfare of 200. Such a choice in terms of inclusion of fixed costs in the welfare objective function is similar to what is done in \cite{ruiz2012}.

\subsection{Unrestriced welfare optimization}\label{sec:unrestricted-welfare}

Notation used throughout the text is provided here for quick reference. The interpretation of any other symbol is given within the text itself.

\subsection*{Notation and Abbreviations}

\medskip

\noindent
Abbreviations:
\begin{tabular}{ll}
MP bids & Stands for bids with either a minimum profit or a maximum payment condition \\ 
MIC bids & Stands for complex orders with a minimum income condition used in OMIE-PCR\\
ITM & Stands for 'in-the-money' \\ 
ATM & Stands for 'at-the-money' \\ 
OTM & Stands for 'out-of-the-money' 
\end{tabular} 

\medskip
\noindent
Sets and indices:

\begin{tabular}{ll}
$i$ & Index for hourly bids, in set $I$ \\ 
$c$ & Index for MP bids, in set $C$ \\
$hc$ & Index for hourly bids associated to the MIC bid $c$, in set $H_c$ \\
$l$ & Index for locations, $l(i)$ (resp. $l(hc)$) denotes the location of bid $i$ (resp. $hc$)\\
$t$ & Index for time slots, $t(i)$ (resp. $t(hc)$) denotes the time slot of bid $i$, (resp. $hc$) \\
$I_{lt} \subseteq I$ & Subset of hourly bids associated to location $l$ and time slot $t$ \\
$HC_{lt} \subseteq HC$  & Subset of MP hourly suborders, associated to location $l$ and time slot $t$ 
\end{tabular} 

\bigskip

Parameters:

\begin{tabular}{ll}
$Q_i, Q_{hc}$ & Power amount of hourly bid $i$ (resp. $hc$), \\ &  $Q<0$ for sell bids, and $Q>0 $ for demand bids \\
$r_{hc} \in [0,1]$ & minimum ratio parameter used to express minimum output levels \\
$P^i, P^{hc}$ & Limit bid price of hourly bid $i$, $hc$\\
$a_{m,k}$ & Abstract linear network representation parameters \\
$w_m$ & Capacity of the network resource $m$ \\
$F_c$ & Start-up or fixed cost associated to bid $c$
\end{tabular} 

\bigskip

Primal decision variables:

\begin{tabular}{ll}
$x_i \in [0,1]$ & fraction of power $ Q_i $ which is executed \\
$x_{hc} \in [0,1]$ & fraction of power $Q_{hc}$ (related to the MIC bid $c$) which is executed\\
$u_c \in \{0,1\}$  & binary variable conditioning the execution or rejection of the MP bid $c$ \\ & (i.e. of the values of $x_{hc}$) \\
$n_k$ & variables used for the abstract linear network representation, related to net export positions \\
\end{tabular} 

\bigskip

Dual decision variables:

\begin{tabular}{ll}
$\pi_{lt}$ & locational uniform price of electricity at location $l$ and time slot $t$ \\
$v_m \geq 0$ & dual variable pricing the network constraint $m$, \\
$s_{i} \geq 0$ & dual variable interpretable as the surplus associated to the execution of bid $i \in I$\\
$s_{hc}^{max} \geq 0$ & dual variable related to the (potential) surplus associated to the execution of bid $hc$\\
$s_{hc}^{min} \geq 0$ & dual variable related to the (potential) surplus associated to the execution of bid $hc$\\
$s_{c} \geq 0$ & dual variable interpretable as the surplus associated to the execution of the MP bid $c$\\

\end{tabular} 

\bigskip

A classical hourly order corresponds to a step of a stepwise offer or demand bid curve relating accepted power quantities to prices. For each such step, the variable $x_i \in [0,1]$ denotes which fraction of this step will be accepted in the market clearing solution. In the same way, variables $x_{hc}$ denote these accepted fractions for bid curves associated to a bid with a minimum profit condition or maximum payment condition (MP bids).

Concerning these MP bids, binary variables $u_c$ are introduced to model the conditional acceptance of a set of hourly bids $hc \in H_c$, controlled via constraints (\ref{primal-eq3}), while constraints  (\ref{primal-eq3b}) enforce minimum acceptance ratios where applicable. They are used for example to model minimum power outputs of power plants. The conditional acceptances will be expressed as price-based decisions (as called in  \cite{zak} \cite{Fernandez-Blanco2015}) using the primal-dual formulation developed in Section \ref{sec-enforced}, involving both quantity and price variables. Parameters $F_c$ correspond to fixed/start-up costs incurred if the MP bid is accepted. Let us also note that a block bid spanning multiple time periods as described in \cite{euphemia,martin,Madani2014} could be described as an MP bid $c$ by using a suitable choice of associated bid curves and minimum acceptance ratios, and setting the corresponding fixed cost parameter $F_c$ to 0 in (\ref{primal-obj}). It turns out that in such a case, minimum profit or maximum payment conditions as described below will exactly correspond to the European market clearing conditions for block orders described in \cite{euphemia,martin,Madani2014}, essentially stating that no loss should be incurred to any accepted block bid, but allowing some block bids to be paradoxically rejected.

Constraint \eqref{primal-eq5} is the balance equation at location $l$ at time $t$, where the right-hand side is the net export position expressed as a linear combination of abstract network elements. Constraint \eqref{primal-eq6} is the capacity constraint of the abstract network resource $m$. This abstract linear network representation covers e.g. DC network flow models or the so-called ATC and Flow-based models used in PCR (see \cite{euphemia}). The usual network equilibrium conditions involving locational market prices apply, as they will be enforced by dual and complementarity conditions (\ref{duala-eq7}), (\ref{cca-eq5}), see \cite{Madani2014}.

The objective function aims at maximizing welfare. For the sake of conciseness, we do not consider ramping constraints of power plants in the main parts of the text, though they can straightforwardly be included in all the developments carried out, as shown in Section \ref{sec:rampingconstr}.

UWELFARE:

\begin{multline}\label{primal-obj}
\max_{x,y,u,n} \sum_{i} (P^{i}Q^{}_{i})x_i + \sum_{c, h\in H_c} (P^{hc}Q^{}_{hc})x_{hc}  - \sum_c F_c u_c
\end{multline}

subject to:

\begin{align}
&x_i \leq 1 & \forall i \in I \ & [s_i]\label{primal-eq1} \\
&x_{hc} \leq u_{c} & \forall h \in H_c, c\in C \ & [s_{hc}^{max}] \label{primal-eq3} \\
&x_{hc} \geq r_{hc }u_{c} & \forall h \in H_c, c\in C \ & [s_{hc}^{min}] \label{primal-eq3b} \\
&u_{c} \leq 1 &\forall c \in C &[s_{c}] \label{primal-eq4} \\
&\sum_{i \in I_{lt}}Q^{}_{i}x_i + \sum_{hc\in HC_{lt}} Q^{}_{hc}x_{hc} \nonumber\\ & \hspace{5cm}   = \sum_{k} e^k_{l,t} n_k, \ &  \forall (l,t) \ \   & [\pi_{l,t}]\label{primal-eq5} \\
&\sum_{k} a_{m,k} n_k  \leq w_{m}\ & \forall m \in N \ \ & [v_{m}]\label{primal-eq6}  \\                  
&x_i, u_c\geq 0, (x_{hc}\ free) \label{primal-eq7} \\
&u \in \mathbb{Z} \label{primal-eq8}
\end{align}

\subsection{Dual and complementarity conditions of the continuous relaxation}

We denote by UWELFARE-CR-DUAL the dual of the continuous relaxation of the welfare maximization program stated above.

UWELFARE-CR-DUAL:

\begin{multline}\label{duala-obj}
\min \sum_{i} s_i  +\sum_c s_c + \sum_m w_m v_m
\end{multline}

\text{subject to: }
\begin{align}
&s_i + Q^{}_i \pi_{l(i),t(i)} \geq Q^{}_i P^{i} , &  \forall i  \qquad [x_i] \label{duala-eq1} \\
&(s_{hc}^{max} - s_{hc}^{min})  + Q^{}_{hc}\pi_{l(hc),t(hc)} = Q^{}_{hc}P^{hc}, & \forall h \in H_c, c \ [x_{hc}] \label{duala-eq2} \\
&s_{c}  \geq \sum_{h \in H_{c}} (s_{hc}^{max} - r_{hc}s_{hc}^{min} ) - F_{c}, &  \forall c \in C  \ [u_{c}] \label{duala-eq6}  \\
&\sum_m a_{m,k} v_m- \sum_{l,t}e^k_{l,t} \pi_{l,t} = 0 & \forall k \in K \ [n_k] \label{duala-eq7}\\
&s_i, s_c, s_{hc}, v_m \geq 0 \label{duala-eq8}
\end{align}

Complementarity conditions: 

\begin{align}
&s_i(1-x_i)=0 &\ \forall i \in I \label{cca-eq1}\\
&s_{hc}^{max}(u_{c}-x_{hc})=0 &\ \forall h,c \label{cca-eq3}\\
&s_{hc}^{min}(x_{hc} - r_{hc}u_{c})=0 &\ \forall h,c \label{cca-eq3b}\\
&s_{c}(1-u_{c})=0 &\forall c \in C \label{cca-eq4} \\
&v_m(\sum_k a_{m,k} n_k - w_m)=0 &\ \forall m \in N \label{cca-eq5}\\
&x_i(s_i  + Q^{}_{i}\pi_{l(i),t(i)} - Q^{}_{i}P^{i})=0 &\ \forall i \in I \label{cca-eq10}\\
&u_{c}(s_{c} - \sum_{h \in H_{c}} (s_{hc}^{max} - r_{hc}s_{hc}^{min} ) + F_{c} ) =0 & \forall c \in C \label{cca-eq15}
\end{align}

%&x_{hc}(s_{hc}^{max} - s_{hc}^{min} + Q^{}_{hc}\pi_{l(hc),t(hc)} - Q^{}_{hc}P^{hc})=0 &\ \forall h,c \label{cca-eq11}\\

As it is well-known, these dual and complementarity conditions, which are optimality conditions for the continuous relaxation of (\ref{primal-obj})-(\ref{primal-eq8}) denoted UWELFARE-CR, exactly describe the nice equilibrium properties we would like to have for a market clearing solution. This could be easily seen from the economic interpretations given in Lemmas \ref{lemma-si},\ref{lemma-surplus-shc}, \ref{lemma-duadur} and Theorem \ref{theorem-mpconditions} below.

Hence, equilibrium and integrality conditions for $u$ cannot be both satisfied unless the continuous relaxation UWELFARE-CR admits a solution which is integral in $u$. In the particular case where there is no fixed cost ($\forall c\in C, F_c=0$), no minimum acceptance ratios ($r_{hc}=0$ for all $hc\in H_c, c \in C$), and there is no condition restraining the conditional acceptances modelled by the binary variables $u_c$ via constraints (\ref{primal-eq3}), it is always optimal to set all $u_c:=1$ and the problem amounts to solving a classical convex market clearing problem where equilibrium can be found which optimizes welfare.

Also, even setting $F_c:=0$ in (\ref{primal-obj}), adding MP conditions to the constraints (\ref{primal-eq1})-(\ref{primal-eq8}), (\ref{duala-eq1})-(\ref{cca-eq15}) to deal with them as in OMIE-PCR (cf. the toy example above with the remark about distinguishable cases, and also Section \ref {sec:comparison-pcr}) would in most cases render the problem infeasible. Hence, equilibrium restrictions must be relaxed, and this can be done in different ways, which is the topic of the next section.

\section{Modelling Near-equilibrium with MP Conditions}\label{sec-modelling-MP}

Section \ref{sec:literature}  reviews previous propositions to handle minimum profit conditions, including the current practice in OMIE-PCR, while Section \ref{sec-enforced} proposes a new approach which seems to be both more appropriate economically speaking, and computationally more efficient. Section \ref{sec:comparison-pcr} makes further technical comparisons between the current OMIE-PCR practice and the new proposition, and recalls an exact linearisation for minimum income conditions used by OMIE proposed in a previous contribution. Ramping conditions are not explicitely considered here, but Section \ref{sec:rampingconstr} shows how all results could be derived when these are included as well in the models.

\subsection{Modelling minimum profit conditions: literature review} \label{sec:literature}

As stated above, when one considers MP conditions or indivisibilities, it is needed to relax market equilibrium conditions to get feasible solutions. A first idea to relax these equilibrium conditions is to relax the complementarity conditions (\ref{cca-eq1})-(\ref{cca-eq15}) while making them satisfied as closely as possible. With the present context and notation, the proposition in \cite{GarciaBertrand} is essentially to minimize the slacks, i.e. the deviations from 0, of the left-hand sides in (\ref{cca-eq1})-(\ref{cca-eq15}), while adding ad-hoc non-convex quadratic constraints guaranteeing non-negative profits for producers, which are then approximated with linear constraints. The idea is generalized in  \cite{gabriel2013} which also considers the possibility of relaxing integrality conditions and to minimize a weighted sum of deviations from complementarity, of deviations from integrality (which could be required to be null), and of uplift variables included in the statement of the minimum profit conditions, corresponding to side payments to ensure revenue adequacy for producers. Leaving aside relaxation of integrality conditions and uplifts, to minimize deviations from complementarity, for each left-hand side expression $g_l \geq 0$,  slack variables $\epsilon_l$ are added together with constraints $\epsilon_l \geq g_l$, and the sum of the $\epsilon_l$ is minimized. Let us note that in the models considered, the fixed costs involved in the minimum profit conditions are not part of the welfare maximizing function in \cite{GarciaBertrand}, while they are included in the welfare in \cite{gabriel2013}.

The model and idea suggested in \cite{gabriel2013} is  considered further in \cite{ruiz2012}, where there is no uplift variable in the statement of minimum profit conditions, therefore requiring revenue adequacy from the uniform market prices only, and where it is observed that minimizing the slacks amounts to minimizing the duality gap given with our notation by (\ref{duala-obj}) minus (\ref{primal-obj}), subject to primal and dual constraints (\ref{primal-eq1})-(\ref{duala-eq8}). The contribution \cite{ruiz2012} observes that this is a significant improvement over the formulation proposed in \cite{gabriel2013}.

In all these propositions, the choice is made to use uniform prices, to ensure minimum profit conditions for producers, and to minimize the deviations from a market equilibrium by minimizing the sum of slacks of all complementarity conditions. In such a case, there is no control on which deviations from market equilibrium are allowed, and in particular, network equilibrium conditions which correspond to optimality conditions of TSOs are often not satisfied.

In the Pan-European PCR market, the choice has been made to ensure network equilibrium conditions as well as equilibrium conditions for all 'classical convex bids' corresponding to steps of classical bid curves. The only allowed deviations from a market equilibrium are that some 'non-convex bids' involving minimum power output constraints or minimum profit (resp. maximum payment) conditions could be paradoxically rejected as in the toy example given above in Section \ref{sec:toyexample}. Let us note that such a 'paradoxical rejection' is also allowed in all other propositions.

Concerning complex bids with a minimum income condition used in OMIE-PCR \cite{euphemia,europex}, minimum profit conditions are of the form:

\begin{equation} \label{mic-eq0}
(u_c=1) \Longrightarrow \sum_{h \in H_c} (-Q^{}_{hc}x_{hc}) \pi_{l(hc),t(hc)} \geq \widetilde{F_c} + \sum_{h \in H_c} (-Q^{}_{hc}x_{hc}) V_c,
\end{equation}

where for the given market prices $\pi_{l,t}$, classical bid curves and the network are 'at equilibrium', describing in particular the fact that ITM hourly bids are fully executed, OTM hourly bids are fully rejected, and ATM hourly bids could be executed or rejected. In the condition, $\widetilde{F_c}$ corresponds to a start-up cost, and $V_c$ to a variable cost of production, while $\sum_{h \in H_c} (-Q^{}_{hc}x_{hc})\pi_{l(hc),t(hc)}$ denotes the revenue generated at the given market prices.

We have shown in a previous article \cite{madani2015mip}, in which other related economic aspects are considered, how to give an exact linearization of this kind of constraints in the whole European market model which can then be formulated as a MILP without \emph{any} auxiliary variables, relying on strong duality for linear programs to enforce equilibrium for the network, classical hourly bids, and hourly bids related to accepted MIC bids. This is reviewed (and extended to include minimum power output level conditions) below in Section \ref{sec:comparison-pcr}. Let us also note here that an exact linearisation similar to the one proposed in \cite{madani2015mip} has been independently proposed  in \cite{Fernandez-Blanco2015}. Though the derivation therein is technically different and e.g. needs to introduce many auxiliary continuous variables and constraints for a McCormick convexification of bilinear binary-continuous terms, a parallel could be made between ideas of the two approaches, which is beyond the scope of the present contribution.

The following Table comparatively summarizes some core characteristics of the previous propositions to model minimum profit conditions and the present one presented below:

\begin{table}[ht]
\begin{center}
\begin{tabular}{c|c|c|c}
Proposition & Start-up costs & Variable costs  & \emph{Strict} spatial  \\ 
  & in the Welfare & in the Min. Profit. Cond. & price equilibrium  \\ 
\hline 
Garcia-Bertrand et al. \cite{GarciaBertrand} & No & marginal costs & No  \\ 
Garcia-Bertrand et al. \cite{GarciaBertrand2} & No & marginal costs & No  \\ 
Gabriel et al. \cite{gabriel2013} & Yes & marginal costs & No  \\ 
Ruiz et al. \cite{ruiz2012} & Yes & marginal costs & No  \\ 
OMIE-PCR  \cite{euphemia} & No & Ad-hoc var. costs & Yes \\ 
Present contribution & Yes & marginal costs & Yes \\ 
\end{tabular} 
\end{center}
\caption{Comparison of  propositions}
\label{table3}
\end{table}

\subsection{A new proposition for modelling MP conditions}\label{sec-enforced}

We use a slightly modified version of a MIP framework introduced in \cite{madani2015mip}, to enforce equilibrium for the convex bids and the network, and which is computationally efficient in particular because it avoids explicitly adding complementarity conditions modelling equilibrium for this convex part, and also any auxiliary variables. It is used to present two distinct models for minimum profit conditions in this setting: one used in practice for many years by OMIE  now coupled to PCR, and the new one involving the 'MP bids' introduced in the present contribution.

\subsubsection{Duality, uniform prices and deviations from equilibrium}\label{subsec:duality-uniformprices}\label{subsec-duality-econ}

Let us consider the primal welfare maximization problem UWELFARE stated in Section \ref{sec:unrestricted-welfare}. Let us now consider a partition $C=C_r \cup C_a$, and the following constraints, fixing all integer variables to some arbitrarily given values (unit-commitment-like decisions):

\begin{align}
&- u_{c_a} \leq -1 \qquad & \forall c_a \in C_a\subseteq C & \hspace{0.4cm} [du^a_{c_a}] \label{fix-eq1} \\
&u_{c_r} \leq 0 \qquad & \forall c_r \in C_r \subseteq C & \hspace{0.4cm}  [du^r_{c_r}] \label{fix-eq2}
\end{align}

Dropping integer constraints (\ref{primal-eq8}) not needed any more, this yields an LP whose dual is:

\begin{multline}\label{dual-obj}
\min \sum_{i} s_i + \sum_c s_c + \sum_m w_m v_m   - \sum_{c_a \in C_a} du^a_{c_a}  
\end{multline}

\text{subject to: }
\begin{align}
&s_i + Q^{}_i \pi_{l(i),t(i)} \geq Q^{}_i P^{i} , &  \forall i  \qquad [x_i] \label{dual-eq1} \\
&(s_{hc}^{max} - s_{hc}^{min}) + Q^{}_{hc}\pi_{l(hc),t(hc)} = Q^{}_{hc}P^{hc}, & \forall h \in H_c, c \ [x_{hc}] \label{dual-eq2} \\
&s_{c_r} + du^r_{c_r} \geq \sum_{h \in H_{c}} (s_{hc}^{max} - r_{hc}s_{hc}^{min} ) - F_{c}, &  \forall c_r \in C_r  \ [u_{c_r}] \label{dual-eq5} \\
&s_{c_a} - du^a_{c_a} \geq \sum_{h \in H_{c}} (s_{hc}^{max} - r_{hc}s_{hc}^{min} ) - F_{c}, &  \forall c_a \in C_a  \ [u_{c_a}] \label{dual-eq6}  \\
&\sum_m a_{m,k} v_m- \sum_{l,t}e^k_{l,t} \pi_{l,t} = 0 & \forall k \in K \ [n_k] \label{dual-eq7}\\
&s_i, s_c, s_{hc},du^r_{c_r},du^a_{c_a}, v_m \geq 0 \label{dual-eq8}
\end{align}

We now write down the complementarity constraints corresponding to these primal and dual programs parametrized by the integer decisions. Economic interpretations are stated afterwards:

\begin{align}
&s_i(1-x_i)=0 &\ \forall i \in I \label{cc-eq1}\\
&s_{hc}^{max}(u_{c}-x_{hc})=0 &\ \forall c,h\in H_c \label{cc-eq3}\\
&s_{hc}^{min}(x_{hc} - r_{hc}u_{c})=0 &\ \forall c, h \in H_c \label{cc-eq3b}\\
&s_{c}(1-u_{c})=0 &\forall c \in C \label{cc-eq4} \\
&v_m(\sum_k a_{m,k} n_k - w_m)=0 &\ \forall m \in N \label{cc-eq5}\\
&(1-u_{c_a})du^a_{c_a}=0 & \forall c_1 \in C_1 \label{cc-eq7} \\
&u_{c_r}du^r_{c_r}=0 & \forall c_r \in C_r  \label{cc-eq6} \\
&x_i(s_i  + Q^{}_{i}\pi_{l(i),t(i)} - Q^{}_{i}P^{i})=0 &\ \forall i \in I \label{cc-eq10}\\
&u_{c_r}(s_{c_r} + du^r_{c_r} - \sum_{h \in H_{c_r}} (s_{hc_r}^{max} - r_{hc_r}s_{hc_r}^{min} ) + F_{c_r}) =0 & \forall c_r \in C_r \label{cc-eq14} \\
&u_{c_a}(s_{c_a} - du^a_{c_a} - \sum_{h \in H_{c_a}} (s_{hc_a}^{max} - r_{hc_a}s_{hc_a}^{min} ) + F_{c_a} ) =0 & \forall c_a \in C_a \label{cc-eq15}
\end{align}

%&x_{hc}(s_{hc}^{max} - s_{hc}^{min}  + Q^{}_{hc}\pi_{l(hc),t(hc)} - Q^{}_{hc}P^{hc})=0 &\ \forall h,c \label{cc-eq11}\\

In what follows, we consider uniform prices, that is all payments depend only and proportionally on a single price $\pi_{l,t}$ for each location $l$ and time period $t$.

\bigskip

In the following Lemmas, it is important to keep in mind the sign convention adopted, according to which a bid quantity $Q>0$ for a buy bid, and $Q<0$ for a sell bid, cf. the description of notation above.

\bigskip

\begin{lemma}[Interpretation of $s_i$ and equilibrium for hourly bids]\label{lemma-si}
Let us consider a solution to (\ref{primal-eq1})-(\ref{primal-eq8}), (\ref{fix-eq1})-(\ref{fix-eq2}), (\ref{dual-eq1})-(\ref{cc-eq15}). Variables $s_i$  correspond to surplus variables, i.e.:

\begin{equation}\label{lemma-surplus-si-eq}
s_i = (Q^{}_{i}P^{i} - Q^{}_{i}\pi_{l(i),t(i)})x_i
\end{equation}

Moreover, the following equilibrium conditions hold, meaning that for the given market prices $\pi_{l,t}$, no other level of execution $x_i^*$ could be preferred to $x_i$:
\begin{itemize}
\item An hourly bid $i$ which is fully executed, i.e. for which $x_i=1$, is ITM or ATM, and the surplus is given by $s_i= (Q^{}_{i}P^{i} - Q^{}_{i}\pi_{l(i),t(i)})x_i = Q^{}_{i}P^{i} - Q^{}_{i}\pi_{l(i),t(i)} \geq 0 $, %, which is strictly positive provided that the bid is ITM and $Q_i\neq0$,

 \item An hourly bid $i$ which is fractionally executed is ATM, i.e. $(Q^{}_{i}P^{i} - Q^{}_{i}\pi_{l(i),t(i)})=0 = s_i$

 \item Fully rejected bids $i$, i.e. for which $x_i=0$, are OTM or ATM, and then $s_i=0$, which also corresponds to the surplus: $s_i=0=(Q^{}_{i}P^{i} - Q^{}_{i}\pi_{l(i),t(i)})x_i = (Q^{}_{i}P^{i} - Q^{}_{i}\pi_{l(i),t(i)})^{+}$,
 
Hence, ITM hourly bids are fully accepted, OTM hourly bids are fully rejected, and ATM hourly bids $i$ can be either accepted or rejected, fully or fractionally.

\end{itemize}

%Conversely, for prices $\pi_{l,t}$ such that in-the-money bids are fully executed and out-of-the-money bids are fully rejected, one can define appropriately $s_i$ as above in order to satisfy constraints indexed by $i$ in (\ref{primal-eq1})-(\ref{primal-eq8}), (\ref{fix-eq1})-(\ref{fix-eq2}), (\ref{dual-eq1})-(\ref{cc-eq15}).
\end{lemma}

\begin{proof}
If $x_i=1$, conditions (\ref{cc-eq10}) ensure that $s_i  =  Q^{}_{i}P^{i} - Q^{}_{i}\pi_{l(i),t(i)} \geq 0$ (since $s_i \geq0$), and the bid is ITM or ATM. Multiplying the obtained equality by $x_i=1$, we get identity (\ref{lemma-surplus-si-eq}).

If $0 < x_i < 1$, $s_i= 0 = s_ix_i$ according to (\ref{cc-eq1}), and (\ref{cc-eq10}) then gives $s_i = Q^{}_{i}P^{i} - Q^{}_{i}\pi_{l(i),t(i)} = 0$: the bid is ATM. Multiplying these equalities by $x_i$, we get identity (\ref{lemma-surplus-si-eq}).

If $x_i=0$, $s_i=0$ according to (\ref{cc-eq1}), which used in dual conditions (\ref{dual-eq1}) gives $Q^{}_{i}P^{i} - Q^{}_{i}\pi_{l(i),t(i)} \leq 0$: the bid is OTM or ATM. As $s_i=x_i=0$, identity (\ref{lemma-surplus-si-eq}) is trivially satisfied.
\end{proof}

\bigskip

\begin{lemma}[Interpretation of $s_{hc}^{max}, s_{hc}^{min} $]\label{lemma-surplus-shc}

Provided that $u_c=1$:

\begin{equation}
(s_{hc}^{max} - r_{hc}s_{hc}^{min} ) =  (Q^{}_{hc}P^{hc} - Q^{}_{hc}\pi_{l(hc),t(hc)})x_{hc} \label{lemma-surplus-shc-eq}
\end{equation}

while if $u_c=0$, then the left-hand side is disconnected from the right-hand side which is 0. Economically speaking, this means that for rejected MP bids, the left-hand side only corresponds to a potential surplus.
\end{lemma}

\begin{proof}
Multiplying (\ref{dual-eq2}) by $x_{hc}$ yields $s_{hc}^{max}x_{hc} - s_{hc}^{min} x_{hc}  = (Q^{}_{hc}P^{hc} - Q^{}_{hc}\pi_{l(hc),t(hc)})x_{hc}$. Using complementarity conditions  (\ref{cc-eq3})-(\ref{cc-eq3b}) where $u_c=1$, according to which $s_{hc}^{max}x_{hc} = s_{hc}^{max}$ and $s_{hc}^{min} x_{hc} = s_{hc}^{min} r_{hc}$,  we get the required identity (\ref{lemma-surplus-shc-eq}). \end{proof}

For rejected MP bids, the sole deviation from an equilibrium affecting the corresponding hourly bids is that some of them could be rejected paradoxically, since at equilibrium, they should or could be rejected if they are out-of-the-money or at-the-money. The situation for accepted MP bids is more interesting. Essentially, the situation is very similar to the case of classical hourly bids described by Lemma \ref{lemma-si}, excepted that here, some 'MP hourly bids' could be incurring a loss due to the minimum acceptance ratio, and several configurations should be distinguished:

\bigskip

\begin{lemma}[Equilibrium and deviations for MP hourly bids of accepted MP bids]
Let us consider hourly bids associated to an accepted MP bid $c$, i.e. such that $u_c=1$. If:
\begin{itemize}
\item $0 \leq r_{hc} < x_{hc} < u_c = 1$, then $s^{max}_{hc} = s^{min}_{hc}=0$ , and the bid $hc$ is at-the-money:

$ (s_{hc}^{max} - r_{hc}s_{hc}^{min} ) = 0  = (Q^{}_{hc}P^{hc} - Q^{}_{hc}\pi_{l(hc),t(hc)}) = (Q^{}_{hc}P^{hc} - Q^{}_{hc}\pi_{l(hc),t(hc)})x_{hc}$

\item $0\leq r_{hc} = x_{hc} < u_c =1$, then $s_{hc}^{max} = 0$ and $(s_{hc}^{max} - r_{hc}s_{hc}^{min} ) = (- r_{hc}s_{hc}^{min} ) = (Q^{}_{hc}P^{hc} - Q^{}_{hc}\pi_{l(hc),t(hc)})x_{hc} \leq 0$. Noting that $s_{hc}^{min} \geq 0$ and $x_{hc} \geq r_{hc} \geq0$, the bid is ATM or OTM, and for $r_{hc} >0$, a loss could be incurred in that case.

\item $0\leq r_{hc} < x_{hc} = u_c =1$, then $s_{hc}^{min} = 0$ and $(s_{hc}^{max} - r_{hc}s_{hc}^{min} ) = s_{hc}^{max} = (Q^{}_{hc}P^{hc} - Q^{}_{hc}\pi_{l(hc),t(hc)})x_{hc} \geq 0$: the bid is ITM or ATM.

\item In the special case where $r_{hc} = 1 = x_{hc} = u_c$, nothing could be inferred on $s_{hc}^{max}, s_{hc}^{min}$, and the bid could be ITM, ATM or OTM, depending on the sign of $ (s_{hc}^{max} - r_{hc}s_{hc}^{min} )$.

\end{itemize}
\end{lemma}

\begin{proof}
This follows a direct discussion of the equality (\ref{lemma-surplus-shc-eq}) of Lemma \ref{lemma-surplus-shc}, using complementarity conditions (\ref{cc-eq3})-(\ref{cc-eq3b}), with $u_c=1$.
\end{proof}

\bigskip

The following Lemma is key to derive Theorem \ref{theorem-mpconditions} and then Corollary \ref{corollary:mpconditions}. These are the main ingredients to derive a MILP formulation avoiding any auxiliary variables of the new model for minimum profit conditions.

\bigskip

\begin{lemma}[Interpretation of $du^a, du^r$] \label{lemma-duadur}
\begin{itemize}
\item[(i)] $\forall c_a \in C_a, du^a_{c_a}$, is an upper bound on the loss of order $c_{a}$, given by \\ $[\sum_{h \in H_{c_a}} (s_{hc_a}^{max} - r_{hc_a}s_{hc_a}^{min} ) - F_{c_a}]^{-} = [\sum_{h \in H_{c_a}} (Q^{}_{hc_a}P^{hc_a} - Q^{}_{hc_a}\pi_{l(hc_a),t(hc_a)})x_{hc_a} - F_{c_a}]^{-}$, where $[a]^{-}$ denotes the negative part of $a$, i.e. $-min[0, a]$.

\item[(ii)] $du^r_{c_r}$ is an upper bound on the sum of the maximum missed individual hourly surpluses (some of which could be negative) minus the fixed cost $F_{c_r}$ of the rejected MP bid $c_r$, that is:

$du^r_{c_r} \geq \sum_{h \in H_{c_r}} (s_{hc_r}^{max} - r_{hc_r}s_{hc_r}^{min} ) -F_{c_r} \geq \sum_{h \in H_{c_r}} (Q^{}_{hc_r}P^{hc_r} - Q^{}_{hc_r}\pi_{l(hc_r),t(hc_r)}) -F_{c_r}$.

\end{itemize}
\end{lemma}

\begin{proof}
(i) Since $u_{c_a}=1$, and using conditions (\ref{cc-eq15}), we have:

$s_{c_a} - du^a_{c_a} = \sum_{h \in H_{c_a}} (s_{hc_a}^{max} - r_{hc_a}s_{hc_a}^{min} ) - F_{c_a}$. Since, $s_{c_a}, du^a_{c_a} \geq 0$, the observation follows (cf. also Lemma \ref{lemma-surplus-shc} for the identity used to replace $(s_{hc_a}^{max} - r_{hc_a}s_{hc_a}^{min} )$).

(ii) Conditions of type (\ref{cc-eq4}) show that $s_{c_r}=0$, which used in (\ref{dual-eq5}) provide the first inequality. Then, as $r_{hc_r} \in [0,1]$ and $s_{hc_r}^{min} \geq 0$ , one has $ (s_{hc_r}^{max} - r_{hc_r}s_{hc_r}^{min} ) \geq (s_{hc_r}^{max} - s_{hc_r}^{min} ) =  Q^{}_{hc}P^{hc} - Q^{}_{hc}\pi_{l(hc),t(hc)}$ where this last equality is given by (\ref{dual-eq2}). The result immediately  follows.
\end{proof}

\begin{theorem}[MP conditions and shadow costs of acceptance $du^a$]\label{theorem-mpconditions}
Let us consider a given partition $C_a \cup C_r$ and a solution to (\ref{primal-eq1})-(\ref{primal-eq8}), (\ref{fix-eq1})-(\ref{fix-eq2}), (\ref{dual-eq1})-(\ref{cc-eq15}):

\begin{itemize}
\item For an accepted sell bid $c_a \in C_a$, i.e. for which $\forall hc_a \in H_{c_a}$, $Q_{hc_a} <0$: 

$ (- \sum_{h \in H_{c_a}} Q^{}_{hc}\pi_{l(hc),t(hc)}x_{hc}) \geq (-\sum_{h \in H_{c_a}} Q^{}_{hc}P^{hc}x_{hc}) + F_{c_a} \Longleftrightarrow du^a_{c_a}=0,$

where the left-hand side of the equivalence expresses that the revenue from trade is greater or equal to the sum of marginal costs plus the fixed cost $F_c$, which is a minimum profit condition.

\item For an accepted buy bid $c_a \in C_a$, i.e. for which $\forall hc_a \in H_{c_a}$, $Q_{hc_a} >0$: 

$( \sum_{h \in H_{c_a}} Q^{}_{hc}\pi_{l(hc),t(hc)}x_{hc}) \leq (\sum_{h \in H_{c_a}} Q^{}_{hc}P^{hc}x_{hc}) - F_{c_a} \Longleftrightarrow du^a_{c_a}=0,$

where the left-hand side of the equivalence expresses that the total payments are lesser or equal to the total utility reduced by the constant term $F_c$, which is a maximum payment condition.

\end{itemize}
\end{theorem}

\begin{proof}
It is a direct consequence of Lemma \ref{lemma-duadur}. If $du^a_{c_a}=0$, then

$\sum_{h \in H_{c_a}} (Q^{}_{hc_a}P^{hc_a} - Q^{}_{hc_a}\pi_{l(hc_a),t(hc_a)})x_{hc_a} - F_{c_a} \geq 0$, which rearranged provides the result (the converse holding as well: if this last inequality holds, the $du^a_{c_a}$ can be set to 0 without altering the satisfaction of the other constraints).
\end{proof}

\begin{corollary}\label{corollary:mpconditions}
MP conditions could be expressed by requiring that shadow costs of acceptance could be set to zero, i.e.:

\begin{equation}
\forall c_a \in C_a,\ du^a_{c_a} = 0
\end{equation}

\end{corollary}

 Naturally, not all MP bid selections $C_a, C_r$ are such that these conditions hold for all accepted MP bids $c_a \in C_a$, cf. e.g. the toy example presented in Section \ref{sec:toyexample}. Moreover, admissible selections $C_a, C_r$ for which all shadow costs of acceptance could be set to zero are not known in advance. However, following \cite{madani2015mip}, we can provide a MILP formulation without any auxiliary variables, exactly describing those admissible partitions $C_a, C_r$, together with a corresponding solution to  (\ref{primal-eq1})-(\ref{primal-eq8}), (\ref{fix-eq1})-(\ref{fix-eq2}), (\ref{dual-eq1})-(\ref{cc-eq15}). This is developed in the next subsection.

\subsubsection{A MILP without auxiliary variables modelling MP conditions}

To state Theorem \ref{maintheorem} about the formulation UMFS, we need to include the following technical constraint limiting the market price range
\begin{equation}\label{pricerangecondition}
\pi_{l,t} \in [-\bar{\pi}, \bar{\pi}] \qquad  \forall l\in L, t\in T.
\end{equation}
$\bar{\pi}$ can be chosen large enough to avoid excluding any relevant market clearing solution (see \cite{Madani2014}). Note that in practice, power exchanges actually do impose that the computed prices $\pi_{l,t}$ stay within a given range in order to limit market power and price volatility, see e.g. \cite{euphemia}.

\textbf{Uniform Market Clearing Feasible Set (UMFS):}

\begin{multline}\label{primaldual-eqobj}
\sum_{i} (P^{i}Q^{}_{i})x_i + \sum_{c, h\in H_c} (P^{hc}Q^{}_{hc})x_{hc}  - \sum_c F_c u_c \\ \geq  \sum_{i} s_i + \sum_c s_c  - \sum_{c \in C} du^a_{c} + \sum_m w_m v_m
\end{multline}

\begin{align}
&x_i \leq 1 & \forall i \in I  [s_i]\label{primaldual-eq1} \\
&x_{hc} \leq u_{c} & \forall h \in H_c, c\in C \ [s_{hc}^{max}] \label{primaldual-eq3} \\
&x_{hc} \geq r_{hc}u_{c} & \forall h \in H_c, c\in C \ [s_{hc}^{min}] \label{primaldual-eq3b} \\
&u_{c} \leq 1 &\forall c \in C [s_{c}] \label{primaldual-eq4} \\
&\sum_{i \in I_{lt}}Q^{}_{i}x_i   + \sum_{hc\in HC_{lt}} Q^{}_{hc}x_{hc} \nonumber\\  & \hspace{ 4cm}= \sum_{k} e^k_{l,t} n_k, \ &  \forall (l,t) \ [\pi_{l,t}]\label{primaldual-eq5} \\
&\sum_{k} a_{m,k} n_k  \leq w_{m}\ & \forall m \in N \ \  [v_{m}]\label{primaldual-eq6}  \\                  
&x, u\geq 0, \label{primaldual-eq7} \\
&u \in \mathbb{Z} \label{primaldual-eq8} \\
&s_i + Q^{}_i \pi_{l(i),t(i)} \geq Q^{}_i P^{i} , &  \forall i  \qquad [x_i] \label{primaldual-eq9} \\
&(s_{hc}^{max} - s_{hc}^{min}) + Q^{}_{hc}\pi_{l(hc),t(hc)} = Q^{}_{hc}P^{hc}, & \forall h \in H_c, c \ [x_{hc}] \label{primaldual-eq10} \\
&s_{c} + du^r_{c} - du^a_{c} \geq \sum_{h \in H_{c}} (s_{hc}^{max} - r_{hc}s_{hc}^{min} ) - F_{c}, &  \forall c \in C  \ [u_{c}] \label{primaldual-eq12} \\
&du^r_{c} \leq M_c (1-u_{c}) &\forall c\in C  \label{primaldual-eq14} \\
&du^a_{c} \leq M_c u_{c} &\forall c\in C  \label{primaldual-eq15} \\
&\sum_m a_{m,k} v_m- \sum_{l,t}e^k_{l,t} \pi_{l,t} = 0 &\ \forall k \in K  [n_k] \label{primaldual-eq17}\\
&s_i,s_c,s_{hc}^{max}, s_{hc}^{min}, du^a, du^r, v_m \geq 0 \label{primaldual-eq18}
\end{align}

\begin{theorem}\label{maintheorem}
(I) Let $(x,u,n,\pi,v,s,du^a, du^r)$ be any feasible point of UMFS satisfying the price range condition (\ref{pricerangecondition}), and let us define $C_r=\{c| u_c=0\},C_a=\{c | u_c=1\}$. 

Then the projection $(x,u,n,\pi,v,s,du^a_{c_a \in C_a}, du^r_{c_r \in C_r})$ satisfies all conditions in (\ref{primal-eq1})-(\ref{primal-eq8}), (\ref{fix-eq1})-(\ref{cc-eq15}). 

(II) Conversely, any point 

$MCS = (x,u,n,\pi,v,s,du^a_{c_a \in C_a}, du^r_{c_r \in C_r})$ feasible for constraints (\ref{primal-eq1})-(\ref{primal-eq8}), (\ref{fix-eq1})-(\ref{cc-eq15}) related to a given arbitrary MIC selection $C = C_r \cup C_a$ which respects the price range condition (\ref{pricerangecondition}) can be ‘lifted’ to obtain a feasible point $\tilde{MCS} = (x,u,n,\pi,v,\tilde{s},\tilde{du^a}, \tilde{du^r})$ of UMFS.
\end{theorem}

\begin{proof}[Sketch of the proof]
This is a straightforward adaptation of Theorem 1 in \cite{madani2015mip}. Essentially:
 (I) any feasible point of UMFS defines a corresponding partition $C_a \cup C_r$ of accepted and rejected MP bids, and conditions (\ref{primaldual-eq12})-(\ref{primaldual-eq15}) ensure that (\ref{dual-eq5})-(\ref{dual-eq6}) are satisfied whatever the corresponding partition is. It is then direct to check that conditions in  (\ref{primal-eq1})-(\ref{primal-eq8}), (\ref{fix-eq1})-(\ref{cc-eq15}) are all satisfied, since, provided (\ref{dual-eq1})-(\ref{dual-eq8}), (\ref{primaldual-eqobj}) can then equivalently be replaced by the complementarity conditions (\ref{cc-eq1})-(\ref{cc-eq15}) as optimality conditions for the program (\ref{primal-obj}) subject to (\ref{primal-eq1})-(\ref{primal-eq8}), (\ref{fix-eq1})-(\ref{fix-eq2}). (Let us note that as $du^a, du^r$ are upper bounds on losses or missed surpluses, see Lemma \ref{lemma-duadur}, the involved big-Ms in (\ref{primaldual-eq14})-(\ref{primaldual-eq15}) are appropriately defined using the technical condition (\ref{pricerangecondition}) bounding the range of possible market prices.)

(II) Conversely, for any partition $C_a \cup C_r$ and a solution to (\ref{primal-eq1})-(\ref{primal-eq8}), (\ref{fix-eq1})-(\ref{cc-eq15}) such that the condition (\ref{pricerangecondition}) is satisfied, we only need to define the additional values $du^a_c = 0$ for $c\in C_r$ and $du^r_c = 0$ for $c \in C_a$. Since the big-Ms have been suitably defined using (\ref{pricerangecondition}), and using (\ref{dual-eq5})-(\ref{dual-eq6}), it is straightforward to check that (\ref{primaldual-eq12})-(\ref{primaldual-eq15}) will be satisfied for all $c \in C$, and hence all conditions (\ref{primaldual-eqobj})-(\ref{primaldual-eq18}) defining UMFS are satisfied (again relying on the equivalence of (\ref{cc-eq1})-(\ref{cc-eq15}) and (\ref{primaldual-eqobj}) as optimality conditions for (\ref{primal-obj}) subject to (\ref{primal-eq1})-(\ref{primal-eq8}), (\ref{fix-eq1})-(\ref{fix-eq2}) provided that  (\ref{primal-eq1})-(\ref{primal-eq8}),(\ref{fix-eq1})-(\ref{fix-eq2})  and the dual conditions (\ref{dual-eq1})-(\ref{dual-eq8}) are satisfied).
\end{proof}

As we want to enforce MP conditions, we need to add to UMFS the following conditions:

\begin{equation}
\forall c \in C,\ du^a_c = 0
\end{equation}

Since we set all the $du^a_c$ to 0, constraints (\ref{primaldual-eq15}) are not needed any more, and constraints (\ref{primaldual-eq12})-(\ref{primaldual-eq14}) reduce to (\ref{mpc-eq11}) below. We hence get the following MILP formulation which we denote 'MarketClearing-MPC', enforcing all MP conditions, and which doesn't make use of any auxiliary variable.

\textbf{MarketClearing-MPC}

\begin{equation}
\max_{} \ \  \sum_{i} (P^{i}Q^{}_{i})x_i + \sum_{c, h\in H_c} (P^{hc}Q^{}_{hc})x_{hc} - \sum_c F_c u_c \label{mpc-obj}
\end{equation}
 subject to:

%&\textbf{PCR-FS} \nonumber\\

\begin{align}   
&\sum_{i} (P^{i}Q^{}_{i})x_i + \sum_{c, h\in H_c} (P^{hc}Q^{}_{hc})x_{hc}  - \sum_c F_c u_c \nonumber\\ & \geq  \sum_{i} s_i + \sum_c s_c  + \sum_m w_m v_m & [\sigma] \label{mpc-eq1}\\
&x_i \leq 1 & \forall i \in I \  [s_i] \label{mpc-eq2}\\
&x_{hc} \leq u_{c} & \forall h \in H_c, c\in C \ [s_{hc}^{max}]  \label{mpc-eq3}\\
&x_{hc} \geq r_{hc }u_{c} & \forall h \in H_c, c\in C \ [s_{hc}^{min}]  \label{mpc-eq3b}\\
&u_{c} \leq 1 &\forall c \in C [s_{c}]  \label{mpc-eq4}\\
&\sum_{i \in I_{lt}}Q^{}_{i}x_i  + \sum_{hc\in HC_{lt}} Q^{}_{hc}x_{hc}  \nonumber\\ & \hspace{4cm} = \sum_{k} e^k_{l,t} n_k, \ &  \forall (l,t) \ \ [\pi_{l,t}] \label{mpc-eq5 }\\
&\sum_{k} a_{m,k} n_k  \leq w_{m}\ & \forall m \in N \ \  [v_{m}]  \label{mpc-eq6}\\                  
&x, u\geq 0,\label{mpc-eq7} \\
&u \in \mathbb{Z} \label{mpc-eq8} \\
&s_i + Q^{}_i \pi_{l(i),t(i)} \geq Q^{}_i P^{i} , &  \forall i  \qquad [x_i]  \label{mpc-eq9}\\
&(s_{hc}^{max} - s_{hc}^{min}) + Q^{}_{hc}\pi_{l(hc),t(hc)} = Q^{}_{hc}P^{hc}, & \forall h \in H_c, c \ [x_{hc}] \label{mpc-eq10} \\
&s_{c} \geq \sum_{h \in H_{c}} (s_{hc}^{max} - r_{hc}s_{hc}^{min} ) - F_c - M_c (1-u_{c}) & \forall c \in C [u_{c}] \label{mpc-eq11} \\
&\sum_m a_{m,k} v_m- \sum_{l,t}e^k_{l,t} \pi_{l,t} = 0 &\ \forall k \in K  [n_k] \label{mpc-eq12}\\
&s_i, s_c,s_{hc}, v_m \geq 0 \label{mpc-eq13}
\end{align}

\subsection{Comparison to 'Minimum income conditions' used by OMIE / PCR}\label{sec:comparison-pcr}

The way minimum profit conditions are handled in OMIE-PCR, described in Section \ref{sec:literature}, presents two substantial differences compared to the MP bids introduced above. First, start-up costs are not included in the welfare maximizing objective function, and second there is the presence of a variable cost $V_c$ which could have no relation to the marginal cost curves described by the hourly bids $hc, c \in H_c$. In \cite{madani2015mip}, we have shown how such 'minimum income conditions' could be linearized exactly without any auxiliary variables, in the frame of the PCR market rules. We adapt here this result to take into account minimum acceptance ratios (corresponding e.g. to minimum output levels) modelled by conditions (\ref{primal-eq3b}), which were not considered in \cite{madani2015mip}. This helps considering more formally the differences between MP bids and classical bids with a minimum income condition currently in use in OMIE-PCR. 

\bigskip

Let us denote by $\widetilde{F_c}$ the actual start-up cost attached to some bid $c$ provided by a producer. As in OMIE-PCR, start-up costs $\widetilde{F_c}$ are not considered in the welfare objective function, it is first needed to set all parameters $F_c=0$ in MarketClearing-MPC, but then, nothing ensures that these start-up costs are recovered for executed bids. It is therefore needed to explicitly include a condition equivalent to (\ref{mic-eq0}), and this can be done in a linear way without any auxiliary variables and any approximation, using the following Lemma:

\bigskip

\begin{lemma}[Adaptation of Lemma 3 in \cite{madani2015mip}]\label{lemma-mic-lin} Consider any feasible point of MarketClearing-MPC in the case where all parameters $F_c$ are set to 0. Then, the following holds:
\begin{equation}\label{mic-eq1}
\forall c \in C, \sum_{h \in H_{c}} (-Q_{hc}x_{hc})\pi_{l(hc),t(hc)} = s_{c} - \sum_{h \in H_{c}} (Q_{hc}P^{hc}) x_{hc}
\end{equation} 
\end{lemma}

\begin{proof}
The identity is trivially satisfied if $u_c=0$, thanks to conditions (\ref{primal-eq3}) and (\ref{cc-eq4}) which are enforced for any feasible point of MarketClearing-MPC.

For $u_c=1$, summing up (\ref{lemma-surplus-shc-eq}) in Lemma \ref{lemma-surplus-shc} over $hc \in H_c$, we get: 

\begin{equation}
\sum_{hc \in H_c} (s_{hc}^{max} - r_{hc}s_{hc}^{min} ) = \sum_{hc \in H_c} (Q^{}_{hc}P^{hc} - Q^{}_{hc}\pi_{l(hc),t(hc)})x_{hc} \label{linearization-dev-eq}
\end{equation}

Then, noting that MarketClearing-MPC enforces (\ref{cc-eq15}) with $du^a=0$, and that we have set all $F_c=0$ not to consider start-up costs in the welfare objective, we can replace the left-hand side of (\ref{linearization-dev-eq}) by $s_c$ to get the required identity. \end{proof}

Let us note that the economic interpretation of the algebraic identity provided by (\ref{mic-eq1}) is straightforward: the total income in the  left-hand side can be decomposed as the total marginal costs plus the total surplus $s_c$ collecting individual surpluses of all the individual bid curves associated to the MIC order.

\bigskip

Using  Lemma \ref{lemma-mic-lin}, the MIC condition (\ref{mic-eq0}) can then be stated in a linear way as follows:

\begin{equation}\label{mic-eq2}
s_c - \sum_{h \in H_{c}} (Q_{hc}P^{hc}) x_{hc} \geq \widetilde{F_c} + \sum_{h \in H_{c}} (-Q_{hc}x_{hc}) V_c -  \overline{M_c}(1-u_c)
\end{equation}

where $\overline{M_c}$ is a fixed number large enough to deactivate the constraint when $u_c=0$. As $u_c=0$ implies $s_c=0$ and $x_{hc}=0$, we set $\overline{M_c}:=\widetilde{F_c}$.

Let us emphasise that once this is done and that we have a linear description of the feasible set handling minimum income conditions as done in OMIE-PCR, many objective functions could be considered, including objective functions involving startup and variable costs in the measure of welfare instead of the marginal costs described by the bid curves associated to a given MIC order.

From a modelling point of view there are therefore two main differences between the MP bids proposed here and the OMIE-PCR MIC orders. The first one is that we need to explicitly state constraints (\ref{mic-eq2}), apart from the single constraint (\ref{mpc-eq1}) that essentially enforce all complementary conditions simultaneously. This is because in the OMIE-PCR model, the fixed and variable costs of the MIC orders are not part of the objective function to be maximised. This is linked to the second difference that in the OMIE-PCR model, there are two different variable costs for MIC orders: one that appears in the objective function to be maximised $P^{hc}$, and another one $V_c$ that appears in the MIC condition (\ref{mic-eq2}). It is questionable whether these two costs actually correspond to real costs of a power plant. This makes the task of regulators in charge of monitoring market behaviour of participants more difficult. Indeed it is not clear any more what is the normal or justifiable market behaviour, and what constitutes gaming or a possible exercise of market power.

\section{Handling ramping constraints}\label{sec:rampingconstr}

Ramping constraints are also called 'load-gradient' conditions in the PCR vocabulary, see \cite{euphemia}. Let us suppose one wants to include in the primal program UWELFARE (\ref{primal-obj})-(\ref{primal-eq8}) ramping constraints for each MP bid representing the technical conditions for operating the corresponding power plant. Our goal is to show how to adapt all results of the present contribution regarding minimum profit (resp. maximum payment) conditions in this setting. Ramping constraints to add are of the form:

\begin{align}
&\sum_{hc \in H_c  | t(hc) = t+1 } (-Q^{hc})x_{hc} - \sum_{hc \in H_c | t(hc) = t} (-Q^{hc})x_{hc} \leq RU_c\ u_c & \forall t \in \{1,...,T-1\}, \forall c \in C \hspace{0.7cm} & [g^{up}_{c,t}] \\
&  \sum_{hc \in H_c | t(hc) = t} (-Q^{hc})x_{hc} - \sum_{hc \in H_c | t(hc) = t+1} (-Q^{hc})x_{hc} \leq RD_c\ u_c & \forall t \in \{1,...,T-1\}, \forall c \in C \hspace{0.7cm}  & [g^{down}_{c,t}]
\end{align}

The occurrences of $u_c$ might seem unnecessary and optional as the conditions would be trivially satisfied for $u_c=0$. However, these occurrences are technically required to derive the appropriate dual program and adapt straightforwardly all previous results. They also make the continuous relaxation of the resulting Integer Program stronger. The corresponding complementarity conditions that will be enforced as all other complementarity conditions in Theorem \ref{maintheorem} are:

\begin{align}
&g^{up}_{c,t}(RU_c\ u_c- \sum_{hc \in H_c  | t(hc) = t } Q^{hc}x_{hc} + \sum_{hc \in H_c | t(hc) = t+1} Q^{hc}x_{hc})=0   & \forall t \in \{1,...,T-1\}, \forall c \in C \hspace{0.7cm} \label{ramping-eq-cc-1} \\
&g^{down}_{c,t}(RD_c\ u_c-\sum_{hc \in H_c | t(hc) = t+1} Q^{hc}x_{hc} + \sum_{hc \in H_c | t(hc) = t} Q^{hc}x_{hc}) =0 & \forall t \in \{1,...,T-1\}, \forall c \in C \hspace{0.7cm} \label{ramping-eq-cc-2} 
\end{align}

Such constraints do not exist for $t=0$ or $t=T$, but the following convention is useful for writing what follows while avoiding distinguishing different cases: $g^{up}_{c,0}=g^{down}_{c,0}=g^{up}_{c,T}=g^{down}_{c,T}=0$.

The dual constraints(\ref{dual-eq2}), (\ref{dual-eq5}) and (\ref{dual-eq6}) should then respectively be replaced by:

\begin{multline}
(s_{hc}^{max} - s_{hc}^{min}) + (Q^{hc} g^{down}_{c,t(hc)-1} - Q^{hc} g^{up}_{c,t(hc)-1} ) + ( Q^{hc} g^{up}_{c,t(hc)} - Q^{hc} g^{down}_{c,t(hc)}) + Q^{}_{hc}\pi_{l(hc),t(hc)} \\ = Q^{}_{hc}P^{hc},  \hspace{0.5cm}\forall h \in H_c, \forall c \in C \ [x_{hc}] \label{ramping-eq-dual-1}
\end{multline}

\begin{multline}
s_{c_r} + du^r_{c_r} \geq \sum_{h \in H_{c_r}} (s_{hc}^{max} - r_{hc}s_{hc}^{min} ) + \sum_t (RU_{c_r} g^{up}_{c_r,t(hc)} + RD_{c_r} g^{down}_{c_r,t(hc)}) - F_{c_r},  \hspace{0.5cm} \forall c_r \in C_r  \ [u_{c_r}]
\end{multline}

\begin{multline}
s_{c_a} - du^a_{c_a} \geq \sum_{h \in H_{c_a}} (s_{hc}^{max} - r_{hc}s_{hc}^{min} ) + \sum_t (RU_{c_a} g^{up}_{c_a,t(hc)} + RD_{c_a} g^{down}_{c_a,t(hc)}) - F_{c_a},  \hspace{0.5cm} \forall c_a \in C_a  \ [u_{c_a}]
\end{multline}

with the corresponding consequence in the formulation of UMFS (used in Theorem \ref{maintheorem}) of replacing  (\ref{primaldual-eq12}) by

\begin{multline}
s_{c} + du^r_{c} - du^a_{c} \geq \sum_{h \in H_{c}} (s_{hc}^{max} - r_{hc}s_{hc}^{min} ) + \sum_t (RU_c g^{up}_{c,t(hc)} + RD_c g^{down}_{c,t(hc)}) - F_{c},  \hspace{0.5cm} \forall c \in C  \ [u_{c}]
\end{multline}

It is shown below that this is all we need to  handle ramping constraints.  Adaptation of Lemma \ref{lemma-duadur} and Theorem \ref{theorem-mpconditions} are then straightforward, as it suffices to replace in the proofs the occurrences of $\sum_{h \in H_{c}} (s_{hc}^{max} - r_{hc}s_{hc}^{min} )$ by its analogue provided by the left-hand side of (\ref{ramping-surplus-income-lin}) below, and the corresponding adaptations needed e.g. in MarketClearing-MPC immediately follows.

These assertions rest on the follwing adaptation of Lemma \ref{lemma-surplus-shc}:

\begin{lemma}[Adaptation of Lemma \ref{lemma-surplus-shc} to handle ramping constraints]

Provided that $u_c=1$, :
\begin{enumerate}
\item 

\begin{center}
\begin{multline}
(s_{hc}^{max} - r_{hc}s_{hc}^{min}) + ( Q^{hc} g^{down}_{c,t(hc)-1} - Q^{hc} g^{up}_{c,t(hc)-1}  )x_{hc} + (Q^{hc} g^{up}_{c,t(hc)} - Q^{hc} g^{down}_{c,t(hc)})x_{hc} \\ =  (Q^{}_{hc}P^{hc} - Q^{}_{hc}\pi_{l(hc),t(hc)})x_{hc} \label{ramping-surplus-individual-id}	
\end{multline}
\end{center}

\item 

\begin{center}
\begin{multline}
\sum_{hc \in H_c}(Q^{hc} g^{down}_{c,t(hc)-1} - Q^{hc} g^{up}_{c,t(hc)-1})x_{hc} + \sum_{hc \in H_c}(Q^{hc} g^{up}_{c,t(hc)} - Q^{hc} g^{down}_{c,t(hc)})x_{hc} \\ = \sum_t (RU_c g^{up}_{c,t(hc)} + RD_c g^{down}_{c,t(hc)})\label{ramping-surplus-id-2}
\end{multline}
\end{center}

\item 

\begin{center}
\begin{equation}
\sum_{h \in H_{c}} (s_{hc}^{max} - r_{hc}s_{hc}^{min} ) + \sum_t (RU_c g^{up}_{c,t(hc)} + RD_c g^{down}_{c,t(hc)}) = \sum_{h \in H_{c}} [Q^{}_{hc}P^{hc} - Q^{}_{hc}\pi_{l(hc),t(hc)}]x_{hc}\label{ramping-surplus-income-lin}
\end{equation}
\end{center}

\end{enumerate}

\end{lemma}

\begin{proof}

\begin{enumerate}
\item is obtained by multiplying equation (\ref{ramping-eq-dual-1}) by the corresponding dual variable $x_{hc}$ and  by using as in Lemma \ref{lemma-surplus-shc} the complementarity conditions  (\ref{cc-eq3})-(\ref{cc-eq3b}) with $u_c=1$, according to which $s_{hc}^{max}x_{hc} = s_{hc}^{max}$ and $s_{hc}^{min} x_{hc} = s_{hc}^{min} r_{hc}$.

\item Summing equations (\ref{ramping-eq-cc-1}) and (\ref{ramping-eq-cc-2}) then summing up over $t$ and rearranging the terms provides the result, noting that it is assumed that $u_c=1$.

\item is a direct consequence of 1. and 2., obtained by summing up (\ref{ramping-surplus-individual-id}) over $hc \in H_c$  and using the identity provided by (\ref{ramping-surplus-id-2})

\end{enumerate}

\end{proof}

\section{A decomposition procedure with Strengthened Benders cuts} \label{sec:Benders}

The contribution in this Section is essentially to show how the Benders decomposition procedure with strengthened cuts described in \cite{Madani2014} for fully indivisible block bids apply to the present context of newly introduced bids with a minimum profit/maximum payment condition (MP bids), providing an efficient method for large-scale instances where both block and MP bids are present, as such decomposition approaches (see also \cite{martin, euphemia}) are known to be efficient to handle block bids. The present extension includes as a special case instances involving block bids with a minimum acceptance ratio as described in \cite{euphemia}. 

This  Benders decomposition procedure solves MarketClearing-MPC , working with (an implicit decription of) the projection $G$ of the MarketClearing-MPC feasible set described by (\ref{mpc-eq1})-(\ref{mpc-eq13}) on the space of primal decision variables $(x_i, x_{hc}, u_c, n_k)$.  In particular, we start with a relaxation of $G$, denoted $G_0$ and described by constraints (\ref{mpc-eq2})-(\ref{mpc-eq8}), and then add Benders cuts to $G_0$ which are valid inequalities for $G$ derived from a so-called worker program until a feasible - hence optimal - solution is found. The worker program generates cuts to cut off incumbents for which no prices exist such that all MP conditions could be enforced, see Theorems \ref{theorem-benders-worker} and \ref{theorem-benders-cuts}. It is shown that these Benders cuts correspond indeed to 'no-good cuts' rejecting the current MP bids combination, see Theorem \ref{theorem-nogood-cuts}. We show how these cuts could be strengthened, providing stronger and sparser cuts which are valid for $G$ when they are computed to cut off solutions which are optimal for the master program (potentially with cuts added at previous iterations where applicable), cf. Theorem \ref{theorem-strenghtened-cuts-globally}. Instead of adding these cuts iteratively after solving the augmented master program each time up to optimality, it could be preferable to generate them \emph{within} a branch-and-cut algorithm solving this master program (hence also MarketClearing-MPC, as MP conditions will be checked and enforced when needed). In that case, the strengthened cuts are locally valid, i.e. in branch-and-bound subtrees originating from incumbents rejected by the worker program during the branch-and-cut algorithm solving the master program, see Theorem \ref{theorem-strenghtened-cuts}. Adding cuts after solving master programs up to optimality is similar to the original approach described in the seminal paper \cite{Benders}, while adding cuts inside a branch-and-cut, which is often more efficient, is sometimes called the 'modern version' of a Benders decomposition. Let us note that the classical Benders cuts of Theorem \ref{theorem-benders-cuts} or their 'no-good' equivalent of Theorem \ref{theorem-nogood-cuts} are always globally valid, as opposed to their strengthened version of Theorem \ref{theorem-strenghtened-cuts}. 

Let us also mention a very interesting result. The revised version of \cite{muller2013competitive} appearing as Chapter 2 in \cite{muller2014auctions} and relying on \cite{Madani2014} proposes an analogue of Theorem \ref{theorem-strenghtened-cuts} in a context which considers general "mixed integer bids", a careful analysis of  which shows they encompass the MP bids proposed here (though there is no mention of applications such as the modeling of start up costs and the minimum profit conditions or ramping constraints, etc). As noted therein, the author generalizes the applicability of the cuts of Theorem 6 in \cite{Madani2014}, similar to those of Theorem \ref{theorem-strenghtened-cuts} below, to these general mixed integer bids (and general convex bids besides) using a completely different technique than the present Benders decomposition which relies on other considerations and the primal-dual formulations presented above (shadow costs of acceptance in Theorem \ref{theorem-mpconditions}, etc).

Let us consider a master branch-and-bound solving (\ref{mpc-obj}) subject to the initial constraints (\ref{mpc-eq2})-(\ref{mpc-eq8}), and let $(x_i^*, x_{hc}^*, u_c^*, n_k^*)$ be an incumbent satisfying (\ref{mpc-eq2})-(\ref{mpc-eq8}) of MarketClearing-MPC. 

A direct application of the Farkas Lemma to the remaining linear conditions (\ref{mpc-eq1}), (\ref{mpc-eq9})-(\ref{mpc-eq13}), which is detailed in appendix, yields:

\bigskip

\begin{theorem}[Worker program of the decomposition]\label{theorem-benders-worker}
Let $(x_i^*, x_{hc}^*, u_c^*, n_k^*)$ be an incumbent satisfying (\ref{mpc-eq2})-(\ref{mpc-eq8}), then there exists $(\pi, s, v)$ such that all MP conditions modelled by (\ref{mpc-eq1}), (\ref{mpc-eq9})-(\ref{mpc-eq13}) are satisfied if and only if:

\begin{multline}\label{test-program-1}
\max_{(x,u,n) \in P} \sum_{i} (P^{i}Q^{}_{i})x_i + \sum_{c, h\in H_c} (P^{hc}Q^{}_{hc})x_{hc} - \sum_c F_c u_c - M_c (1-u_{c}^*)u_c  \\ \leq (\sum_{i} (P^{i}Q^{}_{i})x_i^* + \sum_{c, h\in H_c} (P^{hc}Q^{}_{hc})x_{hc}^* - \sum_c F_c u_c^* ),
\end{multline}

where $P$ is the polyhedron defined by the linear conditions (\ref{mpc-eq2})-(\ref{mpc-eq7}), that is the linear relaxation of (\ref{mpc-eq2})-(\ref{mpc-eq8}). This condition is also equivalent to

\begin{multline}\label{test-program-2}
\max_{(x,u,n) \in P | u_c=0\ if\ u_c^* =0} \ \ \ \  \sum_{i} (P^{i}Q^{}_{i})x_i + \sum_{c, h\in H_c} (P^{hc}Q^{}_{hc})x_{hc} - \sum_c F_c u_c  \\ \leq (\sum_{i} (P^{i}Q^{}_{i})x_i^* + \sum_{c, h\in H_c} (P^{hc}Q^{}_{hc})x_{hc}^* - \sum_c F_c u_c^* ),
\end{multline}
\end{theorem}

where no bigMs are involved.

\begin{proof}
See appendix.
\end{proof}

\medskip

A direct consequence of Theorem \ref{theorem-benders-worker} is:

\medskip

\begin{theorem}[Classical Benders cuts]\label{theorem-benders-cuts}
Suppose $(x_i^*, x_{hc}^*, u_c^*, n_k^*)$ doesn't belong to $G$, i.e. there are no prices such that MP conditions could all be satisfied, i.e. for which the test of Theorem \ref{theorem-benders-worker} fails.

Then, the following Benders cut is a valid inequality for $G$ and cuts off the current incumbent $(x_i^*, x_{hc}^*, u_c^*, n_k^*)$:

\begin{multline}\label{benders-cut}
\sum_{i} (P^{i}Q^{}_{i})x_i^{\#} + \sum_{c, h\in H_c} (P^{hc}Q^{}_{hc})x_{hc}^{\#} - \sum_c F_c u_c^{\#} - M_c (1-u_{c})u_c^{\#}  \\ \leq (\sum_{i} (P^{i}Q^{}_{i})x_i + \sum_{c, h\in H_c} (P^{hc}Q^{}_{hc})x_{hc} - \sum_c F_c u_c ),
\end{multline}

where $(x_i^{\#}, x_{hc}^{\#}, u_c^{\#}, n_k^{\#})$ is an optimal solution to the left-hand side worker 
program in (\ref{test-program-1}) (resp. (\ref{test-program-2})).
\end{theorem}

\bigskip

\begin{lemma}\label{lemma-welfare-univocally}
In the feasible set of MarketClearing-MPC, welfare is univocally determined by an MP bids combination, i.e., by given arbitrarily values for the variables $u_c$.
\end{lemma}

\begin{proof}
Let us consider a feasible point of MarketClearing-MPC and the corresponding MP bids combination $C_a \cup C_r$. As detailed in Theorem \ref{maintheorem} and its proof, this point is then feasible for (\ref{primal-eq1})-(\ref{primal-eq8}), (\ref{fix-eq1})-(\ref{fix-eq2}), (\ref{dual-eq1})-(\ref{dual-eq8}) and (\ref{cc-eq1})-(\ref{cc-eq15}), which are optimality conditions for the welfare maximization program (\ref{primal-obj})-(\ref{primal-eq8}), (\ref{fix-eq1})-(\ref{fix-eq2}) where only the integer values of the variables $u_c$ have been fixed.
\end{proof}

\bigskip

\begin{observation}\label{observation-1}
An optimal solution of the left-hand side of (\ref{test-program-1}) is always such that $u_c^{\#} = 0$ if $u_c^{*} = 0$, because of the penalty coefficients $M_c$, or alternatively because $u_c^{\#} $ corresponds to the optimal dual variable of (\ref{mpc-eq11}) which is not tight when $u_c^{*} = 0$.
\end{observation}

\bigskip

\begin{theorem}[No-good / Combinatorial Benders cuts]\label{theorem-nogood-cuts}
Suppose $(x_i^*, x_{hc}^*, u_c^*, n_k^*)$ doesn't belong to $G$, i.e. there are no prices such that MP conditions could all be satisfied, i.e. for which the test of Theorem \ref{theorem-benders-worker} fails.

Then, the following 'no-good cut' is a valid inequality for $G$ and cuts off the current incumbent:

\begin{equation}\label{cut-nogood}
\sum_{c | u_c^* = 1} (1-u_c) +  \sum_{c | u_c^* = 0} u_c \geq 1,
\end{equation}

basically excluding the current MP bids combination.
\end{theorem}

\begin{proof}
This is a direct consequence of Theorem \ref{theorem-benders-cuts}. Suppose we need to cut off $(x_i^*, x_{hc}^*, u_c^*, n_k^*)$ by adding (\ref{benders-cut}). For any other solution $(x_i, x_{hc}, u_c, n_k)$ such that $u_c = u_c^*$ for all $c \in C$, the left-hand side value of (\ref{benders-cut}) will trivially be the same as with $u^*$. The right-hand side will also be the same as with $u^*$ according to Lemma \ref{lemma-welfare-univocally}, because welfare is univocally determined by the values of the $u_c$. Hence any such solution will also violate (\ref{benders-cut}) and it is therefore needed to change the value of at least one of the $u_c$, providing the result.
\end{proof}

\begin{theorem}[Globally valid strengthened Benders cuts]\label{theorem-strenghtened-cuts-globally}

Let $(x_i^*, x_{hc}^*, u_c^*, n_k^*)$  be an optimal solution for the master program (\ref{mpc-obj}) subject to (\ref{mpc-eq2})-(\ref{mpc-eq8}),  potentially with additional valid inequalities. If the test of Theorem \ref{theorem-benders-worker} fails, the following sparse cut is a valid inequality for $G$:

\begin{equation}\label{cut-strengthened-globally}
\sum_{c | u_c^* = 1} (1-u_c) \geq 1,
\end{equation}

meaning that at least one of the currently accepted MP bids should be excluded in any valid market clearing solution satisfying MP conditions.
\end{theorem}

\begin{proof}
This is also a consequence of Theorem \ref{theorem-benders-cuts}. First, observe that (\ref{cut-strengthened-globally}) trivially implies (\ref{cut-nogood}) and hence cuts off $(x_i^*, x_{hc}^*, u_c^*, n_k^*)$, according to Theorem \ref{theorem-nogood-cuts}. It remains to show that it is also a valid inequality for $G$.

Let $(x_i^*, x_{hc}^*, u_c^*, n_k^*)$ be the optimal solution considered that violates (\ref{benders-cut}), i.e., such that:

\begin{multline}\label{locally-derivation-2}
\sum_{i} (P^{i}Q^{}_{i})x_i^{\#} + \sum_{c, h\in H_c} (P^{hc}Q^{}_{hc})x_{hc}^{\#} - \sum_c F_c u_c^{\#} - M_c (1-u_{c}^*)u_c^{\#}  \\ > (\sum_{i} (P^{i}Q^{}_{i})x_i^* + \sum_{c, h\in H_c} (P^{hc}Q^{}_{hc})x_{hc}^* - \sum_c F_c u_c^* ),
\end{multline}

which using Observation \ref{observation-1} reduces to:

\begin{multline}\label{locally-derivation-3}
(\sum_{i} (P^{i}Q^{}_{i})x_i^* + \sum_{c, h\in H_c} (P^{hc}Q^{}_{hc})x_{hc}^* - \sum_c F_c u_c^* )  \\ < 
\sum_{i} (P^{i}Q^{}_{i})x_i^{\#} + \sum_{c, h\in H_c} (P^{hc}Q^{}_{hc})x_{hc}^{\#} - \sum_c F_c u_c^{\#} 
\end{multline}

Suppose $(x_i, x_{hc}, u_c, n_k)$ is feasible for (\ref{mpc-eq2})-(\ref{mpc-eq8}) (with the potential added valid inequalities obtained at previous iterations). 
Because of the optimality of $(x_i^*, x_{hc}^*, u_c^*, n_k^*)$, 

\begin{multline}\label{locally-derivation}
(\sum_{i} (P^{i}Q^{}_{i})x_i + \sum_{c, h\in H_c} (P^{hc}Q^{}_{hc})x_{hc} - \sum_c F_c u_c )  \\ \leq (\sum_{i} (P^{i}Q^{}_{i})x_i^* + \sum_{c, h\in H_c} (P^{hc}Q^{}_{hc})x_{hc}^* - \sum_c F_c u_c^* ) \\ < \sum_{i} (P^{i}Q^{}_{i})x_i^{\#} + \sum_{c, h\in H_c} (P^{hc}Q^{}_{hc})x_{hc}^{\#} - \sum_c F_c u_c^{\#} 
\end{multline}

Now suppose $(x_i, x_{hc}, u_c, n_k)$ does not satisfy (\ref{cut-strengthened-globally}), i.e., that $\sum_{c | u_c^* = 1} (1-u_c) = 0$. Then, also using Observation \ref{observation-1} exactly as to reduce (\ref{locally-derivation-2}) to (\ref{locally-derivation-3}), the valid cut (\ref{benders-cut}) that this other solution must satisfy to potentially be in $G$ reduces to:

\begin{multline}
\sum_{i} (P^{i}Q^{}_{i})x_i^{\#} + \sum_{c, h\in H_c} (P^{hc}Q^{}_{hc})x_{hc}^{\#} - \sum_c F_c u_c^{\#}  \\ \leq (\sum_{i} (P^{i}Q^{}_{i})x_i + \sum_{c, h\in H_c} (P^{hc}Q^{}_{hc})x_{hc} - \sum_c F_c u_c ),
\end{multline}

which contradicts (\ref{locally-derivation}). Hence, (\ref{cut-strengthened-globally}) must hold for any other $(x_i, x_{hc}, u_c, n_k)$ that is in $G$.
\end{proof}

Now, suppose we want to use the sparse cuts of Theorem \ref{theorem-strenghtened-cuts-globally} \emph{within} the  branch-and-bound tree solving the master program, instead of adding them after solving up-to-optimality the master program (together with the cuts obtained at previous iterations where applicable). Then these cuts are valid \emph{locally}, i.e. in the subtrees originating from the incumbents to cut off, as their validity depends on the local optimality of this incumbent to cut off:

\bigskip

\begin{theorem}[Locally valid strengthened Benders cuts] \label{theorem-strenghtened-cuts}

Let again $(x_i^*, x_{hc}^*, u_c^*, n_k^*)$  be an incumbent obtained via an LP relaxation at a given node of the branch-and-cut solving the master program (\ref{mpc-obj}) subject to (\ref{mpc-eq2})-(\ref{mpc-eq8}). If the test of Theorem \ref{theorem-benders-worker} fails, the following sparse cut is locally valid, i.e. is valid in the subtree of the branch-and-bound originating from the current node providing the incumbent $(x_i^*, x_{hc}^*, u_c^*, n_k^*)$:

\begin{equation}
\sum_{c | u_c^* = 1} (1-u_c) \geq 1,
\end{equation}

meaning that at least one of the currently accepted MP bids should be excluded in any solution found deeper in the subtree.
\end{theorem}

\begin{proof}
This is also a consequence of Theorem \ref{theorem-benders-cuts} and the proof is a slight variant of the proof of Theorem \ref{theorem-strenghtened-cuts-globally}. Since in the present case $(x_i^*, x_{hc}^*, u_c^*, n_k^*)$ is just an incumbent and no longer globally optimal for the master program, to reproduce the argument providing (\ref{locally-derivation}), we use the \emph{local} optimality of the incumbent $(x_i^*, x_{hc}^*, u_c^*, n_k^*)$ obtained via an LP relaxation, and the fact that the other solutions considered $(x_i, x_{hc}, u_c, n_k)$ lie in the subtree originating from the current node providing $(x_i^*, x_{hc}^*, u_c^*, n_k^*)$.
\end{proof}

\section{Numerical Experiments}\label{sec-numerical}
Implementation of the models and algorithms proposed above have been made in Julia using JuMP\cite{LubinDunningIJOC}, an open source package providing an algebraic modelling language embedded within Julia, CPLEX 12.6.2 as the underlying MIP solver, and ran on a laptop with an i5 5300U CPU with 4 cores @2.3 Ghz and 8GB of RAM. The source code and sample data sets used to compute the tables presented below are available online, see \cite{revisiting_mp_conditions} . Let us note that thanks to Julia/JuMP, it is easy to switch from one solver to another, provided that all the required features are available.  
\emph{Raw implementations} of the primal-dual formulation MarketClearing-MPC, and the classic and modern Benders decompositions all fit within 250 lines of code including input-output data management (see the file 'dam.jl' provided online), while some solution checking tools provided in an auxiliary file span about 180 lines of code.

Our main purpose here is to compare the new approach proposed to the market rules used until now by the power exchange OMIE (part of PCR). We thus have considered realistic datasets corresponding to the case of Spain and Portugal. Notable differences compared to real data for example available at \cite{omie_public_data} is that the marginal costs of the first steps of each bid curve associated to a given MIC order have been replaced by the variable cost of that MIC order whenever they were below the variable cost, and as a consequence, a minimum acceptance ratio of 0.6 has been set for the first step of each of these bid curves. The rationale for such modifications is the following: marginal costs for the first steps of the bid curves are sometimes very low (even almost null) certainly to ensure a reasonable level of acceptance of the corresponding offered quantities for operational reasons, and increasing them would decrease too much the accepted quantities at some hours, which is counterbalanced by setting an appropriate acceptance ratio at each hour in case the MP order is part of the market outcome solution. Let us recall that an MP order can only be accepted if the losses incurred at some hours (due to the minimum acceptance ratios forcing paradoxical acceptances  which are 'measured' by the dual variables $s_{hc}^{min}$) are sufficiently compensated by the profits made at some other hours of the day. All the costs  have then been uniformly scaled to obtain interesting instances where e.g. the MP conditions are not all verified if only the primal program (\ref{primal-obj})-(\ref{primal-eq8}) is solved. As network aspects are not central here, a simple two nodes network corresponding to coupling Spain and Portugal is considered.

\begin{table}[ht]
\begin{center}
\begin{tabular}{c|c|c|c|c|c|c|c}
Inst.	&Welfare	&Abs. gap	&Solver's cuts	&Nodes	&Runtime	&\# MP Bids	&\# Curve Steps\\
\hline
1	&151218658.27	&0.00	&24	&388	&72.63	&92	&14494\\
2	&115365156.34	&0.00	&15	&181	&38.08	&90	&14309\\
3	&112999837.94	&1644425.79	&21	&4085	&600.17	&91	&14329\\
4	&107060355.83	&0.00	&16	&0	&7.63	&89	&14370\\
5	&100118316.52	&0.00	&15	&347	&96.06	&89	&15091\\
6	&97572068.18	&0.00	&18	&1116	&143.65	&86	&14677\\
7	&87937471.32	&1091700.74	&27	&4958	&600.11	&87	&14979\\
8	&89866979.23	&0.00	&87	&1707	&296.41	&87	&16044\\
9	&86060320.81	&0.00	&97	&361	&57.27	&81	&15177\\
10	&90800596.61	&3755055.95	&59	&2430	&600.02	&90	&16475\\
\end{tabular} 
\end{center}
\caption{Instances with 'MIC Orders' as in OMIE-PCR}
\label{table4}
\end{table}

\begin{table}[ht]
\begin{center}
\begin{tabular}{c|c|c|c|c|c|c|c}
Inst.	&Welfare	&Abs. gap	&Solver's cuts	&Nodes	&Runtime	&\# MP Bids	&\# Curve Steps\\
\hline
1	&151487156.16	&0.00	&11	&9	&17.36	&92	&14494\\
2	&115475592.36	&0.00	&11	&0	&16.38	&90	&14309\\
3	&114220400.20	&0.00	&24	&0	&17.23	&91	&14329\\
4	&107219935.90	&0.00	&35	&7	&17.48	&89	&14370\\
5	&100743738.16	&0.00	&14	&0	&14.74	&89	&15091\\
6	&98359291.45	&0.00	&10	&0	&15.67	&86	&14677\\
7	&89251699.16	&0.00	&84	&3	&22.92	&87	&14979\\
8	&90797407.15	&0.00	&27	&0	&21.58	&87	&16044\\
9	&86403721.22	&0.00	&35	&7	&25.04	&81	&15177\\
10	&94034444.59	&0.00	&20	&0	&19.58	&90	&16475\\
\end{tabular} 
\end{center}
\caption{Instances with MP bids - MarketClearing-MPC formulation}
\label{table5}
\end{table}

As both market models, though different, pursue the same goal of modelling start-up costs and marginal costs recovery conditions while representing in some ways indivisibilities of production (with minimum acceptance ratios or using very low marginal costs for the first amounts of power produced in some original datasets), Tables \ref{table4} \& \ref{table5} propose a comparison from a computational point of view, which shows the benefits of the new approach. A key issue wit the current practice is the absence of the fixed costs in the objective function and the occurrence of an 'ad-hoc' variable cost in the minimum income conditions which is not related to the marginal costs used in the objective function. The objective function in the continuous relaxations somehow 'goes in a direction' which may not be the most appropriate with respect to the enforcement of the minimum income conditions. On the other side, the new approach seems more natural as it enforces minimum profit conditions by requiring that the 'shadow costs of acceptance' $du^a$ must all be null, see Corollary \ref{corollary:mpconditions}.

Table \ref{table6} is to be compared with Table \ref{table5} e.g. in terms of runtimes and visited nodes, as it solves exactly the same market model. Heuristics of the solver have been here deactivated as primal feasible solutions found need to be obtained as optimal solutions of the LP relaxation at the given node for the local cuts of Theorem \ref{theorem-strenghtened-cuts} to be valid (cf. its statement above). As it can be seen, the Benders decomposition is faster by an order of magnitude for the instances at hand.

\begin{table}[ht]
\begin{center}
\begin{tabular}{c|c|c|c|c|c|c|c}
Inst.	&Welfare	&Lazy cuts	&Solver's cuts	&Nodes	&Runtime	&\# MP Bids	&\# Curve Steps\\
\hline
1	&151487156.16	&2	&0	&5	&2.66	&92	&14494\\
2	&115475592.36	&1	&18	&5	&1.38	&90	&14309\\
3	&114220400.20	&1	&28	&3	&1.81	&91	&14329\\
4	&107219935.90	&2	&14	&11	&1.78	&89	&14370\\
5	&100743738.16	&1	&12	&3	&1.36	&89	&15091\\
6	&98359291.45	&1	&3	&3	&1.36	&86	&14677\\
7	&89251699.16	&1	&29	&8	&1.54	&87	&14979\\
8	&90797407.15	&1	&11	&3	&1.66	&87	&16044\\
9	&86403721.22	&2	&1	&13	&2.24	&81	&15177\\
10	&94034444.59	&1	&40	&4	&1.54	&90	&16475\\
\end{tabular} 
\end{center}
\caption{Instances with MP bids - Benders decomposition of Theorem \ref{theorem-strenghtened-cuts}}
\label{table6}
\end{table}

\section{Conclusions}

A new approach to minimum profit or maximum payment conditions has been proposed in the form of a bidding product called 'MP bid', which turns out to generalize both block orders with a minimum acceptance ratio used in France, Germany or Belgium,  and, mutatis mutandis, complex orders with a minimum income condition used in Spain and Portugal. The corresponding market clearing conditions such as minimum profit or maximum payment conditions can be expressed with a 'primal-dual' MILP model involving both primal decision variables such as unit commitment or power output variables, and dual decision variables such as market prices or economic surpluses of market participants, while avoiding the introduction of \emph{any} auxiliary variables, whether continuous or binary. Moreover, it can be used to derive a Benders decomposition with strengthened cuts of a kind which is known to be efficient to handle block bids. These MP bids hence seem an appropriate tool to foster market design and bidding products convergence among the different regions which form the coupled European day-ahead electricity markets of the Pan-European PCR project. Also, compared to the MIC orders currently in use at OMIE-PCR, they have the following additional advantages. Firstly, they lead to optimisation models that can be solved much more quickly. Secondly, the proposed instruments seem to be more aligned with the operating constraints and cost structure of the power plants that they are supposed to represent in the market. Finally, they are more natural (from an economic point of view) and simpler (from a modelling point of view), leading to a market model easier to understand for participants and monitor for regulators. All the models and algorithms have been implemented in Julia/JuMP and are available online together with sample datasets to foster research and exchange on the topic. The models and algorithms can also be used to clear instances involving block bids only (small extensions could also be added to handle linked and exclusive block orders as described in \cite{euphemia} if desired). European day-ahead electricity markets will certainly be subject to a major evolution in the coming years, as many challenges are still to be faced, which calls for further research within the academic and industrial communities. The present contribution is a proposal made in that frame.

\textbf{Acknowledgements:}
We greatly thank OMIE and N-Side for providing us with data used to generate realistic instances. This text presents research results of the P7/36 PAI project COMEX, part of the Interuniversity Attraction Poles (IAP) Programme of the Belgian Science Policy Office.This work was also supported by EC-FP7-PEOPLE MINO Marie-Curie Initial Training Network (grant number 316647) and by EC-FP7 COST Action TD1207. The scientific responsibility is assumed by the authors.

%%%%%%%%%%%%%%%%%%%%%%%%%%%%%%%%%%
%%%%%%%%%%%%%%%%%%%%%%%%%%%%%%%%
%%%%%%%%%%%%%%%%%%%%%%%%%%%%%%%%%%%%%%

%\paragraph{Paragraph headings} Use paragraph headings as needed.
%
%% For one-column wide figures use
%\begin{figure}
%% Use the relevant command to insert your figure file.
%% For example, with the graphicx package use
%%  \includegraphics{example.eps}
%\end{figure}
%%
%% For two-column wide figures use
%\begin{figure*}
%% Use the relevant command to insert your figure file.
%% For example, with the graphicx package use
%%  \includegraphics[width=0.75\textwidth]{example.eps}
%\end{figure*}
%\begin{table}
%\begin{tabular}{lll}
%\hline\noalign{\smallskip}
%\end{tabular}
%\end{table}
%

\appendix

\section{Omitted proofs in main text}

\subsection{Proof of Theorem \ref{theorem-benders-worker}}
Reminder of the Farkas Lemma \cite{schrijver}, which is used in the proof afterward: 

$\exists x : Ax<=b, x\geq 0$ if and only if $\forall y: y\geq 0, yA\geq 0 \Rightarrow yb \geq0$

\begin{proof}

Applying the Farkas lemma, given an incumbent $(x_i^*, x_{hc}^*, u_c^*, n_k^*)$, a solution $(s_i, s_{hc}^{max}, s_{hc}^{min}, s_c, \pi_{l,t}, v_m)$ to the remaining linear conditions (\ref{mpc-eq1}), (\ref{mpc-eq9})-(\ref{mpc-eq13}) exist if and only if:

\begin{multline}\label{farkas-condition-1}
\sum_{i} (P^{i}Q^{}_{i})x_i + \sum_{c, h\in H_c} (P^{hc}Q^{}_{hc})x_{hc} - \sum_c F_c u_c - M_c (1-u_{c}^*)u_c  \\ \leq \sigma (\sum_{i} (P^{i}Q^{}_{i})x_i^* + \sum_{c, h\in H_c} (P^{hc}Q^{}_{hc})x_{hc}^* - \sum_c F_c u_c^* )
\end{multline}

$\forall (\sigma, x_i, x_{hc}, u_c, n_k)$ such that:

\begin{align}
&x_i \leq \sigma & \forall i \in I \  [s_i] \\
&x_{hc} \leq u_{c} & \forall h \in H_c, c\in C \ [s_{hc}^{max}]  \\
&x_{hc} \geq r_{hc} u_{c} & \forall h \in H_c, c\in C \ [s_{hc}^{min}]  \\
&u_{c} \leq \sigma &\forall c \in C [s_{c}]  \\
&\sum_{i \in I_{lt}}Q^{}_{i}x_i  + \sum_{hc\in HC_{lt}} Q^{}_{hc}x_{hc}  \nonumber\\ & \hspace{4cm} = \sum_{k} e^k_{l,t} n_k, \ &  \forall (l,t) \ \ [\pi_{l,t}] \label{farkas-condition-6} \\
&\sum_{k} a_{m,k} n_k  \leq w_{m}\ & \forall m \in N \ \  [v_{m}] \label{farkas-condition-7} \\
& x_i, x_{hc}, u_c, \sigma \geq 0 \label{farkas-condition-end}
\end{align}

Since the condition described by (\ref{farkas-condition-1})-(\ref{farkas-condition-end}) is trivially satisfied when $\sigma=0$ (technically assuming that network conditions (\ref{farkas-condition-6})-(\ref{farkas-condition-7}) could be satisfied when $x_i=x_{hc}=0$), we can normalize, i.e. set $\sigma:=1$ and the condition becomes 

\begin{multline}\label{test-ineq}
\max_{\forall (x_i, x_{hc}, u_c, n_k) \in P}\sum_{i} (P^{i}Q^{}_{i})x_i + \sum_{c, h\in H_c} (P^{hc}Q^{}_{hc})x_{hc} - \sum_c F_c u_c - M_c (1-u_{c}^*)u_c  \\ \leq \sum_{i} (P^{i}Q^{}_{i})x_i^* + \sum_{c, h\in H_c} (P^{hc}Q^{}_{hc})x_{hc}^* - \sum_c F_c u_c^*,
\end{multline}

where $P$ is the polyhedron defined by the linear conditions (\ref{mpc-eq2})-(\ref{mpc-eq7}), that is the linear relaxation f (\ref{mpc-eq2})-(\ref{mpc-eq8}). This provides the first result (\ref{test-program-1}).

Now, observe that an optimal solution of the left-hand side of (\ref{test-program-1}) or (\ref{test-ineq}) is always such that $u_c^{\#} = 0$ if $u_c^{*} = 0$, because of the penalty coefficients $M_c$, or alternatively because $u_c^{\#} = 0$ corresponds to the optimal dual variable of (\ref{mpc-eq11}) which is not tight when $u_c^{*} = 0$. This proves (\ref{test-program-2}). \end{proof}

\bibliography{template}

\bibliographystyle{plainnat}

\end{document}